\documentclass{amsart}
\usepackage[utf8]{inputenc}

\usepackage[whole]{bxcjkjatype}

\usepackage{amssymb}
\usepackage{mathtools}
\mathtoolsset{showonlyrefs=true}
\usepackage{amsthm}
\usepackage{graphicx}
\usepackage{bm}
\usepackage{mathrsfs}
\usepackage{url}
\usepackage{tikz}
\usepackage{tkz-euclide}
\usepackage{pgfplots}
\usetikzlibrary{quotes,angles,calc}

\numberwithin{equation}{section}

\theoremstyle{definition}
\newtheorem{alg}{Algorithm}[section]
\newtheorem{ex}[alg]{Example}

\theoremstyle{plain}
\newtheorem{thm}[alg]{Theorem}
\newtheorem{cor}[alg]{Corollary}
\newtheorem{lem}[alg]{Lemma}
\theoremstyle{remark}
\newtheorem{rem}[alg]{Remark}

\newcommand{\av}{\operatorname{av}}
\newcommand{\sgn}{\operatorname{sgn}}

\renewcommand{\Im}{\operatorname{Im}}

\allowdisplaybreaks
\begin{document}
	\title[Curve-shortening scheme for moving boundary problems]{A fully discrete curve-shortening polygonal evolution law for moving boundary problems}
	\author{Koya Sakakibara}
	\address[K.~Sakakibara]{Department of Applied Mathematics, Faculty of Science, Okayama University of Science, 1-1 Ridaicho, Kita-ku, Okayama-shi, Okayama 700-0005, Japan; RIKEN iTHEMS, 2-1 Hirosawa, Wako-shi, Saitama 351-0198, Japan}
	\email{ksakaki@math.kyoto-u.ac.jp}
	\author{Yuto Miyatake}
	\address[Y.~Miyatake]{Cybermedia Center, Osaka University, Japan, 1-32 Machikaneyama, Toyonaka, Osaka 560-0043, Japan}
	\email{miyatake@cas.cmc.osaka-u.ac.jp}

\begin{abstract}
We consider the numerical integration of moving boundary problems with the curve-shortening property, such as the mean curvature flow and Hele-Shaw flow. 
We propose a fully discrete curve-shortening polygonal evolution law. 
The proposed evolution law is fully implicit, and the key to the derivation is to devise the definitions of tangent and normal vectors and tangential and normal velocities at each vertex in an implicit manner. 
Numerical experiments show that the proposed method allows the use of relatively large time step sizes and also captures the area-preserving or dissipative property in good accuracy.
\end{abstract}
\maketitle

\section{Introduction}
\label{sec:introduction}

In this paper, we are concerned with the numerical integration of a smooth Jordan curve $\Gamma^t$ for $t\in[0,\infty)$ which evolves in the plane.
By using its $S^1$ parameterization $\bm{x}^t=\bm{x}^t(u)$ ($u\in[0,1]$), the motion of $\Gamma^t$ is described by
\begin{equation}
	\partial_t\bm{x}^t=v^t\bm{n}^t+w^t\bm{t}^t,
	\quad
	t > 0,
	\label{eq:evolution_law}
\end{equation}
where $\bm{t}^t$ and $\bm{n}^t$ respectively represent the unit tangent and unit outward normal vectors of $\Gamma^t$, and $v^t$ and $w^t$ are the normal and tangential velocities.
Both $\Gamma^t$ and $\bm{x}^t$ are called a curve.
A problem regarding the evolution of the curve $\Gamma^t$ is called a moving boundary problem.
Most planar curves have their characteristics in terms of the length and enclosed area, which are defined by
\begin{equation}
	\mathcal{L}[\Gamma^t] := \int_{\Gamma^t}\, \mathrm ds,
	\quad
	\mathcal{A}[\Omega^t] := \int_{\Omega^t}\,\mathrm d \bm{x},
	\label{leng_area_func}
\end{equation}
respectively, where $\Omega = \Omega[\Gamma]$ denotes the interior simply-connected region bounded by $\Gamma$.
We, in particular, focus on problems whose length decreases monotonically.
This property is called the curve-shortening property.
Such a problem can be seen as a gradient flow of the length functional $\mathcal{L}[\Gamma^t]$ with respect to some underlying space.
We further restrict our attention to problems that do not generate any topological changes.
Typical examples include the mean curvature flow~\cite{M56}, area-preserving mean curvature flow~\cite{G86}, and Hele-Shaw flow~\cite{H98}.
We note that it was proved in~\cite{GH86,G86,GP16,G87} that the first two examples do not allow any topological changes during the evolution, and numerical observations suggested that the same property holds for the Hele-Shaw flow~\cite{SY19}.

One approach to computing the evolution law \eqref{eq:evolution_law} is to approximate the boundary by a polygonal or polynomial curve to find a semi-discrete evolution law, and then to integrate it by applying a numerical integrator such as a Runge--Kutta method (see, e.g.~\cite{Y15,KBS16}).
These approaches are often called direct approaches, and these would be preferred for problems without any topological changes.
Note that if a flow of interest admits topological changes,
indirect approaches, such as the phase-field method and the level set method, would be preferred~\cite{JQ09,P18}, but such cases are out of the scope of the present paper.
There are mainly two difficulties for direct approaches: distributions of vertices and time step size restriction.
\begin{itemize}
\item In the continuous scenarios, the tangential velocity $w^t$ does not affect the shape of the evolving curve $\Gamma^t$~\cite{EG87}.
It only affects the parameterization of the curve $\Gamma^t$.
However, the tangential velocity plays a crucial role in discrete settings.
For example, if we set the discrete tangential velocities to $0$, there could be portions where a distribution of vertices is dense or sparse, which could make it difficult to solve the problem over long times.
\item General-purpose methods such as explicit Runge--Kutta methods cannot inherit the curve-shortening property,
which could lead to unstable or non-physical behavior.
To avoid these scenarios, we are usually forced to use sufficiently small time step sizes.
In most cases, the time step size must be of $\mathrm{O}(N^{-2})$, where $N$ denotes the number of vertices.
\end{itemize}

The first difficulty can be overcome by setting the discrete tangential velocities carefully.
The asymptotic uniform distribution method, which arranges the vertices uniformly, and the curvature adjusted method, which arranges the vertices densely where the absolute value of the curvature is large, are practically useful (see, e.g.~\cite{HLS94, K94, K97, MS04, MS06, BGN11, SY11, SY13}).

In this paper, we shall focus on the second difficulty.
To compute the moving boundary problems stably with relatively large time step sizes, we employ the idea of geometric numerical integration~\cite{HLW06}.
Namely, we intend to derive a fully discrete evolution law that guarantees the curve-shortening property independently of the time step sizes.
To achieve this goal, we consider applying the so-called discrete gradient method~\cite{G96,QC96,QM08,H10,CH11,HL14,MB16} to \eqref{eq:evolution_law}
(see also its extensions to partial differential equations~\cite{FU99,MA08,CM11,FM11,CG12,MM14}).
The method has been developed to derive energy-dissipative schemes for gradient systems and energy-preserving schemes for Hamiltonian systems.
However, the method usually requires a concrete expression of gradient systems.
Though the moving boundary problems have concrete expressions as gradient systems, these depend on the underlying spaces and could be complicated.
Moreover, taking the tangential velocities into account further complicates the derivation of the intended evolution law.

In this paper, we show that the intended curve-shortening fully discrete evolution law can be systematically constructed independently of the underlying space  even if  the concrete expression as a gradient system is not specified.
The derived evolution law is curve-shortening independently of the discrete tangential velocities, which indicates that the tangential velocities can be defined in view of the aforementioned asymptotic uniform distribution method or curvature adjusted method.
We note that it seems impossible to derive a fully discrete evolution law that inherits both the curve-shortening and area-preserving (or dissipative) properties simultaneously.
Discussions for area-preserving or dissipative evolution laws are given in Appendix \ref{sec:area}, where canonical Runga--Kutta methods play an essential role.

The contents of this paper are as follows.
In Section~\ref{sec:moving_boundary_problem}, we give a general description of the moving boundary problem \eqref{eq:evolution_law}, show fundamental properties for the length and enclosed area and provide several examples.
In Section~\ref{sec:polygonal_moving_boundary_problem}, we discretize the moving boundary problem in space, where the smooth curve $\bm{x}^t$ of the original problem is discretized by using a polygonal curve.
In Section~\ref{sec:time-discretized_polygonal_moving_boundary_problem}, which is the main section, we describe the intended fully discrete evolution law.
Numerical experiments are given in Section~\ref{sec:numerical_experiments}, and concluding remarks are given in Section~\ref{sec:concluding_remarks}.

\section{Moving boundary problems}
\label{sec:moving_boundary_problem}
	
This section reviews some fundamental properties of the moving boundary problem \eqref{eq:evolution_law} and shows several examples. 
	
\subsection{Fundamental properties of moving boundary problems}
For a smooth curve $\Gamma^t$ with $S^1$ parameterization $\bm{x}^t=\bm{x}^t(u)$ ($u\in[0,1]$), the unit tangent vector $\bm{t}^t$ and the unit outward normal vector $\bm{n}^t$ are respectively defined by
\begin{equation}
\bm{t}^t=\partial_{s(\bm{x}^t)}\bm{x}^t=\frac{\partial_u\bm{x}^t}{|\partial_u\bm{x}^t|},
\quad
\bm{n}^t=-J\bm{t}^t.
\label{eq:unit_tangent_unit_outward_normal}
\end{equation}
Here, $\partial_{s(\bm{x})}$ denotes the differential operator with respect to the arclength parameter of the curve $\bm{x}$, and $J$ is the rotation matrix of angle $\pi/2$:
\begin{equation}
	\partial_{s(\bm{x})}\mathsf{F}
	:=
	\frac{1}{|\partial_u\bm{x}|}\partial_u\mathsf{F},
	\quad
	J=
	\begin{pmatrix}
		0 & -1\\
		1 & 0
	\end{pmatrix}.
	\label{eq:arclength_parameter_derivative_rotation}
\end{equation}
Here and hereafter, in this section, the symbol $\mathsf{F}$ represents a general function defined on a suitable curve.
We note that, as already mentioned in Section~\ref{sec:introduction}, the normal velocity $v^t$ is essential in the formulation \eqref{eq:evolution_law} because it determines the dynamics of the curve $\Gamma^t$ and the tangential velocity $w^t$ only affects the parameterization of the curve $\Gamma^t$~\cite{EG87}.

The following theorem characterizes how the length and enclosed area \eqref{leng_area_func} evolve under the evolution law \eqref{eq:evolution_law}.
	
\begin{thm}[see, e.g.~\cite{G93,K08}]
\label{thm:dtdL_dtdA_dtdG}
%Suppose that a smooth Jordan curve $\Gamma^t$ evolves according to the evolution law \eqref{eq:evolution_law}.
%Then, we have the following relations:
For the smooth Jordan curve $\Gamma^t$ that evolves according to the evolution law \eqref{eq:evolution_law}, it follows that
\begin{equation}
	\partial_t\mathcal{L}[\Gamma^t]
    =
    \int_{\Gamma^t}\kappa^tv^t\,\mathrm{d}s,
    \quad
    \partial_t\mathcal{A}[\Omega^t]
    =
    \int_{\Gamma^t}v^t\,\mathrm{d}s,
	\label{eq:dtdL_dtdA_dtdG}
\end{equation}
where $\kappa^t$ denotes the curvature of $\Gamma^t$ in $-\bm{n}^t$ direction.
\end{thm}
	
We here sketch the proof because mimicking the proof plays an important role when discretizing the evolution law \eqref{eq:evolution_law}.
	
\begin{proof}
First, we consider the length.
A direct calculation yields that 
\begin{align}
	\partial_t\mathcal{L}[\Gamma^t]
	&=
	\int_0^1\partial_t|\partial_u\bm{x}^t|\,\mathrm{d}u
	=
	\int_0^1\frac{\partial_u\bm{x}^t}{|\partial_u\bm{x}^t|}\cdot\partial_t\partial_u\bm{x}^t\,\mathrm{d}u
	=
	-\int_0^1 \partial_u\left(\frac{\partial_u\bm{x}^t}{|\partial_u\bm{x}^t|}\right)\cdot\partial_t\bm{x}^t\,\mathrm{d}u.
	\label{eq:dtdL_calculation}
\end{align}
The definitions of the unit tangent and unit outward normal vectors \eqref{eq:unit_tangent_unit_outward_normal} and the Frenet formulae  $\partial_{s(\bm{x}^t)}\bm{t}^t= -\kappa^t\bm{n}^t$ and $\partial_{s(\bm{x}^t)}\bm{n}^t = \kappa^t\bm{t}^t$ imply that
\begin{align}
	\partial_u\left(\frac{\partial_u\bm{x}^t}{|\partial_u\bm{x}^t|}\right)
	=
	\partial_u\bm{t}^t
	=
	-\kappa^t|\partial_u\bm{x}^t|\bm{n}^t.
	\label{eq:Frenet_without_arclength}
\end{align}
Hence we obtain
\begin{equation}
	\partial_t\mathcal{L}[\Gamma^t]
	=
	\int_0^1\kappa^t|\partial_u\bm{x}^t|\bm{n}^t\cdot(v^t\bm{n}^t+w^t\bm{t}^t)\,\mathrm{d}u
	=
	\int_{\Gamma^t}\kappa^tv^t\,\mathrm{d}s.
\end{equation}
		
Next, we consider the enclosed area.
Using the Frenet formulae, we have
\begin{align}
    \partial_t|\partial_u\bm{x}^t|
    &=
    \frac{\partial_u\bm{x}^t}{|\partial_u\bm{x}^t|}\cdot\partial_t\partial_u\bm{x}^t
    =
    \bm{t}^t\cdot\partial_u(v^t\bm{n}^t+w^t\bm{t}^t)\\
    &=
    |\partial_u\bm{x}^t|\bm{t}^t\cdot(\partial_{s(\bm{x}^t)}v^t\bm{n}^t+v^t\partial_{s(\bm{x}^t)}\bm{n}^t+\partial_{s(\bm{x}^t)}w^t\bm{t}^t+w^t\partial_{s(\bm{x}^t)}\bm{t}^t)\\
    &=
    (\kappa^tv^t+\partial_{s(\bm{x}^t)}w^t)|\partial_u\bm{x}^t|.
    \label{eq:dtdr_dtdn1}
\end{align}
A similar computation shows that
\begin{align}
    \partial_t\bm{t}^t
    &=
    -\frac{\partial_t|\partial_u\bm{x}^t|}{|\partial_u\bm{x}^t|}\partial_u\bm{x}^t+\frac{\partial_t\partial_u\bm{x}^t}{|\partial_u\bm{x}^t|}\\
    &=
    -\frac{\partial_u\bm{x}^t}{|\partial_u\bm{x}^t|}(\kappa^tv^t+\partial_{s(\bm{x}^t)}w^t)+\partial_{s(\bm{x}^t)}(v^t\bm{n}^t+w^t\bm{t}^t)\\
    &=
    (\partial_{s(\bm{x}^t)}v^t-\kappa^tw^t)\bm{n}^t.
\end{align}
Rotating this by angle $-\pi/2$ gives
\begin{align}
    \partial_t\bm{n}^t
	=
	-(\partial_{s(\bm{x}^t)}v^t-\kappa^tw^t)\bm{t}^t.
	\label{eq:dtdr_dtdn2}
\end{align}
Using \eqref{eq:dtdr_dtdn1} and \eqref{eq:dtdr_dtdn2}, we obtain
\begin{align}
	\partial_t\mathcal{A}[\Omega^t]
	&=
	\frac{1}{2}\int_0^1\partial_t(\bm{x}^t\cdot\bm{n}^t|\partial_u\bm{x}^t|)\,\mathrm{d}u\\
	&=
	\frac{1}{2}\int_0^1\left[(\partial_t\bm{x}^t\cdot\bm{n}^t+\bm{x}^t\cdot\partial_t\bm{n}^t)|\partial_u\bm{x}^t|+\bm{x}^t\cdot\bm{n}^t\partial_t|\partial_u\bm{x}^t|\right]\,\mathrm{d}u\\
	&=
	\frac{1}{2}\int_{\Gamma^t}\left[v^t-(\partial_{s(\bm{x}^t)}v^t-\kappa^tw^t)\bm{x}^t\cdot\bm{t}^t+(\kappa^tv^t+\partial_{s(\bm{x}^t)}w^t)\bm{x}^t\cdot\bm{n}^t\right]\,\mathrm{d}s.
\end{align}
The proof is completed by performing integration by parts:
\begin{align}
    \int_{\Gamma^t}
    \partial_{s(\bm{x}^t)}v^t\bm{x}^t\cdot\bm{t}^t\,\mathrm{d}s
    &=
    -\int_{\Gamma^t}v^t(\partial_{s(\bm{x}^t)}\bm{x}^t\cdot\bm{t}^t+\bm{x}^t\cdot\partial_{s(\bm{x}^t)}\bm{t}^t)\,\mathrm{d}s\\
    &=
    -\int_{\Gamma^t}v^t(1-\kappa^t\bm{x}^t\cdot\bm{n}^t)\,\mathrm{d}s,\\
    \int_{\Gamma^t}
    \partial_{s(\bm{x}^t)}w^t\bm{x}^t\cdot\bm{n}^t\,\mathrm{d}s
    &=
    -\int_{\Gamma^t}w^t(\partial_{s(\bm{x}^t)}\bm{x}^t\cdot\bm{n}^t+\bm{x}^t\cdot\partial_{s(\bm{x}^t)}\bm{n}^t)\,\mathrm{d}s\\
    &=
    -\int_{\Gamma^t}w^t\kappa^t\bm{x}^t\cdot\bm{t}^t\,\mathrm{d}s.
\end{align}

\end{proof}
	
The relation \eqref{eq:Frenet_without_arclength} can be interpreted in such a way that the most left-hand side determines the curvature and the unit outward normal vector simultaneously.
This viewpoint will be essential in discretizing the evolution law \eqref{eq:evolution_law}.
	
\subsection{Examples}
	
We give three examples: the mean curvature flow, area-preserving mean curvature flow and Hele-Shaw flow. 
They differ in the definition of the normal velocity $v^t$.
For all cases, the length functional decays monotonically, which is called the curve-shortening property.
The enclosed area is decreasing for the mean curvature flow, while it is preserving for the latter two flows. 
		
\begin{ex}[mean curvature flow]
\label{ex:mean_curvature_flow}
The mean curvature flow is the simplest moving boundary problem, which was originally proposed in~\cite{M56} to describe an ideal grain boundary motion in two dimensions.
This flow can also be understood as a model of the motion of a super elastic rubber band, with a small mass in a viscous medium~\cite{G86}.
The velocity of the mean curvature flow is given by 
\begin{equation}
	v^t=-\kappa^t,
	\quad
	t>0.
	\label{eq:mean_curvature_flow}
\end{equation}
Substituting this expression into the relations \eqref{eq:dtdL_dtdA_dtdG}, we obtain the curve-shortening property:
\begin{equation}
	\partial_t\mathcal{L}[\Gamma^t]
	=
	-\int_{\Gamma^t}(\kappa^t)^2\,\mathrm{d}s
	< 0
	\label{eq:dtdL_mean_curvature_flow}
\end{equation}
and the area-dissipative property:
\begin{equation}
	\partial_t\mathcal{A}[\Omega^t]
	=
	-\int_{\Gamma^t}\kappa^t\,\mathrm{d}s
	=
	-2\pi<0.
	\label{eq:MCF_area_speed}
\end{equation}
%Namely, the length decays monotonically, and the enclosed area decays at constant speed $-2\pi$.
It is also well known that any smooth Jordan curve which evolves according to the mean curvature flow becomes convex in a finite time~\cite{G87} and further shrinks to a point~\cite{GH86}.

From the first relation in \eqref{eq:dtdL_dtdA_dtdG}, the curvature $\kappa^t$ can be understood as the first variation of the length functional.
Thus the mean curvature flow is the $L^2$ gradient flow of the length functional on the space of smooth planar curves.
\end{ex}
	
\begin{ex}[area-preserving mean curvature flow]
\label{ex:area-preserving_mean_curvature_flow}
The area-preserving mean curvature flow, which is a model of a super elastic rubber band surrounding an incompressible fluid, is a modification of the mean curvature flow with the area-preservation constraint~\cite{G86}.
The normal velocity of this flow is given by
\begin{equation}
	v^t=-\kappa^t+\langle\kappa^t\rangle_{\Gamma^t},
	\quad
	t>0,
\end{equation}
where $\langle\mathsf{F}\rangle_\Gamma$ denotes the average of a function $\mathsf{F}$ on a curve $\Gamma$:
\begin{equation}
	\langle\mathsf{F}\rangle_\Gamma
	=
	\frac{1}{\mathcal{L}[\Gamma]}\int_\Gamma\mathsf{F}\,\mathrm{d}s.
\end{equation}
For this flow, we have the curve-shortening property:
\begin{align}
	\partial_t\mathcal{L}[\Gamma^t]
	&=
	\int_{\Gamma^t}\kappa^t(-\kappa^t+\langle\kappa^t\rangle_{\Gamma^t})\,\mathrm{d}s
	=
	-\int_{\Gamma^t}(\kappa^t)^2\,\mathrm{d}s+\frac{1}{\mathcal{L}[\Gamma^t]}\left(\int_{\Gamma^t}\kappa^t\,\mathrm{d}s\right)^2\\
	&=
	\frac{1}{\mathcal{L}[\Gamma^t]}\left(\left(\int_{\Gamma^t}\kappa^t\,\mathrm{d}s\right)^2-\int_{\Gamma^t}1^2\,\mathrm{d}s\int_{\Gamma^t}(\kappa^t)^2\,\mathrm{d}s\right)
	\le
	0
\end{align}
and the area-preserving property:
\begin{align}
	\partial_t\mathcal{A}[\Gamma^t]
	=
	\int_{\Gamma^t}\left(-\kappa^t+\langle\kappa^t\rangle_{\Gamma^t}\right)\,\mathrm{d}s
	=
	0,
\end{align}
where the Cauchy--Schwarz inequality is used.
The area-preserving mean curvature flow is the $L^2$ gradient flow of the length functional along curves which enclose a fixed area.
\end{ex}
	
\begin{ex}[Hele-Shaw flow]
\label{ex:Hele-Shaw_flow}
The Hele-Shaw flow, which was originally studied by experiments in~\cite{H98}, describes the motion of viscous fluid in a quasi-two-dimensional space.
The normal velocity of this flow is given by
\begin{equation}
	v^t=-\nabla p^t\cdot\bm{n}^t, \quad t>0,
	\label{eq:Hele-Shaw0}
\end{equation}
where $p^t$ is the solution to the Laplace equation
\begin{equation}
	\begin{dcases*}
		\triangle p^t=0&in $\Omega^t$, $t>0$,\\
		p^t=\sigma\kappa^t&on $\Gamma^t$, $t>0$.
	   \end{dcases*}
	        \label{eq:Hele-Shaw}
\end{equation}
Here, $\sigma$ denotes the surface tension coefficient  (see~\cite{L93,GV06} for details).
It readily follows from \eqref{eq:dtdL_dtdA_dtdG}, \eqref{eq:Hele-Shaw0} and \eqref{eq:Hele-Shaw}  that the Hele-Shaw flow satisfies the curve-shortening property:
\begin{align}
	\partial_t\mathcal{L}[\Gamma^t]
	&=
	-\frac{1}{\sigma}\int_{\Gamma^t}p^t\nabla p^t\cdot\bm{n}^t\,\mathrm{d}s
	=
	-\frac{1}{\sigma}\int_{\Omega^t}\nabla\cdot(p^t\nabla p^t)\,\mathrm{d}\bm{x}
	=
	-\frac{1}{\sigma}\int_{\Omega^t}|\nabla p^t|^2\,\mathrm{d}\bm{x}
	\le
	0
\end{align}
and the area-preserving property:
\begin{equation}
	\partial_t\mathcal{A}[\Gamma^t]
	=
	-\int_{\Gamma^t}\nabla p^t\cdot\bm{n}^t\,\mathrm{d}s
	=
	-\int_{\Omega^t}\nabla\cdot\nabla p^t\,\mathrm{d}\bm{x}
	=
	0.
\end{equation}
The Hele-Shaw flow can be regarded as the $\Im(\Lambda_{\Gamma^t})$ gradient flow of the length functional on the space of functions of which the integral on the curve is equal to $0$ (see, e.g.~\cite{K08,KTY13} for details).
Here, $\Lambda_\Gamma$ denotes the Dirichlet-to-Neumann operator associated with a curve $\Gamma$.
\end{ex}
	
\section{Semi-discrete polygonal moving boundary problems}	\label{sec:polygonal_moving_boundary_problem}
	
In this section, we discretize the moving boundary problem \eqref{eq:evolution_law} in space.
The discretization is based on~\cite{SY19,SY13}.
Below we give definitions of discrete tangent vectors, discrete outward normal vectors, discrete curvatures, discrete tangential velocities and discrete normal velocities.
Although the contents of this section are not new, we give the details as preliminaries to Section~\ref{sec:time-discretized_polygonal_moving_boundary_problem}.

\subsection{Semi-discrete evolution law}	
Let $\Gamma^t$ be an $N$-sided polygonal Jordan curve:
\begin{equation}
	\Gamma^t=\bigcup_{i=1}^N\Gamma_i^t,
	\quad	
	\Gamma_i^t=(\bm{X}_{i-1}^t,\bm{X}_i^t):=\{(1-\lambda)\bm{X}_{i-1}^t+\lambda\bm{X}_i^t\mid\lambda\in(0,1)\},
\end{equation}
where $\bm{X}_i^t \in \mathbb{R}^2$ denotes the coordinate of the $i$th vertex of $\Gamma^t$.
Here and hereafter, for quantities $\{\mathsf{F}_i\}_{i=1}^N$ associated with an $N$-sided polygonal Jordan curve, we adopt the periodic notation: we always assume that $\mathsf{F}_{i}:=\mathsf{F}_{i \text{ mod } N}$ (in particular, $\mathsf{F}_0:=\mathsf{F}_N$ and $\mathsf{F}_{N+1}:=\mathsf{F}_1$).
Moreover, we also denote a polygonal curve $\Gamma$ as $\Gamma(\bm{X})$ whenever we want to emphasize that the vertices of $\Gamma$ are $\{\bm{X}_i\}_{i=1}^N$.

Let us consider a semi-discrete evolution law of the form
\begin{equation}
	\partial_t\bm{X}_i^t
	=
	V_i^t\bm{N}_i^t+W_i^t\bm{T}_i^t,
	\quad
	i=1,2,\ldots,N,\ 
	t > 0,
	\label{eq:polygonal_evolution_law}
\end{equation}
where $\bm{T}_i^t$ and $\bm{N}_i^t$ respectively represent the unit tangent and unit outward normal vectors of $\Gamma^t$ at the vertex $\bm{X}_i^t$.
Definitions of $\bm{N}_i^t$, $\bm{T}_i^t$, $V_i^t$ and $W_i^t$ are given in the subsequent subsections.

In what follows, we use lower case letters for functions or quantities on edges.
For example, the unit tangent and unit outward normal vectors on the $i$th edge are denoted by $\bm{t}_i^t$ and $\bm{n}_i^t$, respectively.

\subsection{Definitions of the unit tangent vector, unit outward normal vector and curvature}	\label{subsec:tangent_normal_curvature_pmbp}
	
In this subsection, we give definitions of several quantities that appear in the semi-discrete evolution law \eqref{eq:polygonal_evolution_law}.
Given the vertices $\{\bm{X}_i^t\}_{i=1}^N$, the unit tangent vectors $\{\bm{t}_i^t\}_{i=1}^N$  and unit outward normal vectors $\{\bm{n}_i^t\}_{i=1}^N$ on edges are defined straightforwardly.
Using $\{\bm{t}_i^t\}_{i=1}^N$ and $\{\bm{n}_i^t\}_{i=1}^N$, we define unit tangent vectors $\{\bm{T}_i^t\}_{i=1}^N$ and unit outward normal vectors $\{\bm{T}_i^t\}_{i=1}^N$ on vertices.
Though the velocities $\{v_i^t\}_{i=1}^N$ on edges need to be set for each problem, we give the relation between the velocities $\{V_i^t\}_{i=1}^N$ on vertices and $\{v_i^t\}_{i=1}^N$ on edges.
Furthermore, we define the curvatures $\{\kappa_i^t\}_{i=1}^N$ on edges.
After giving these definitions, we show a discrete analogue of Theorem \ref{thm:dtdL_dtdA_dtdG}.

%	There are several possibilities on the definitions of the unit tangent vectors $\bm{T}_i^t$, the unit outward normal vectors $\bm{N}_i^t$, and the curvatures on edges of the polygonal curve $\Gamma^t$. 

We introduce a discrete version of the differential operator \eqref{eq:arclength_parameter_derivative_rotation}.
The difference operator $\partial_{\Gamma,i}$ that operates on a quantity $\{\mathsf{F}_i\}_{i=1}^N$ on vertices is defined by
\begin{equation}
	\partial_{\Gamma,i}\mathsf{F}
	:=
	\frac{\mathsf{F}_i-\mathsf{F}_{i-1}}{r_i},
	\quad
	i=1,2,\ldots,N,
\end{equation}
where $r_i=|\bm{X}_i-\bm{X}_{i-1}|$ denotes the length of the $i$th edge $\Gamma_i$.
Then, the unit tangent vector $\bm{t}_i^t$ and the unit outward normal vector $\bm{n}_i^t$ on the $i$th edge are defined straightforwardly by
\begin{equation}
	\bm{t}_i^t=\partial_{\Gamma^t,i}\bm{X}^t,
	\quad
	\bm{n}_i^t=-J\bm{t}_i^t,
	\quad
	i=1,2,\ldots,N.
	\label{def:semid_t}
\end{equation}
We next introduce the average operator $\av_{\Gamma,i}$ that operates on a quantity $\{\mathsf{f}_i \}_{i=1}^N$ on edges.
This operator is defined by
\begin{equation}
	\av_{\Gamma,i}\mathsf{f}
	=
	\frac{\mathsf{f}_i+\mathsf{f}_{i+1}}{2\cos(\varphi_i/2)},
	\quad
	i=1,2,\ldots,N,
\end{equation}
where $\varphi_i=\theta_{i+1}-\theta_i$, and $\theta_i$ denotes the $i$th tangent angle of $\Gamma$: $\bm{t}_i=(\cos\theta_i,\sin\theta_i)^{\mathrm{T}}$ (see Figure~\ref{fig:T-N_polygonal_curve}).
Using the average operator, we define the unit tangent vector $\bm{T}_i^t$ and the unit outward normal vector $\bm{N}_i^t$ at the $i$th vertex $\bm{X}_i^t$ by
\begin{equation}
	\bm{T}_i^t
	:=
	\av_{\Gamma^t,i}\bm{t}^t
	=
	\begin{pmatrix}
			\cos(\theta_i^t+\varphi_i^t/2)\\
			\sin(\theta_i^t+\varphi_i^t/2)
	\end{pmatrix},
	\quad
	\bm{N}_i^t
	:=
	-J\bm{T}_i^t,
	\quad
	i=1,2,\ldots,N.
	\label{eq:tangent_normal_pmbp}\end{equation}
For more details, see, e.g.~\cite{SY19}.
We readily obtain the following lemma.

\begin{figure}[tb]
\centering
\begin{tikzpicture}
\coordinate (O0) at (1.5,0);
\node[circle,fill,scale=0.5,label=left:{$\bm{X}_{i-1}$}] (O1) at (0,0) {};
\node[circle,fill,scale=0.5,label={[xshift=-0.3cm, yshift=-0.6cm]$\bm{X}_i$}] (O2) at (-1,3) {};
\node[circle,fill,scale=0.5,label=below:{$\bm{X}_{i+1}$}] (O3) at (-6,4) {};
\coordinate (O4) at (0.1,4.1);
\coordinate (O5) at (-1.4,4.2);
%\coordinate (O6) at (-1.9,4.1);
    \tkzInCenter(O2,O3,O5)
    \tkzGetPoint{O6}
%\coordinate (O6) at ($5*(O3)!.5!(O5)$);
\coordinate (O7) at ($(O2)!0.4!-90:(O6)$) ;
\draw[dotted,thick] (O0) -- (O1);
\draw[thick] (O1) -- (O2) -- (O3);
%\draw[very thick,->] (O2) -- ($(O2)!1.2cm!(O4)$) node[right] {$\bm{N}_i$};
%\draw[thick,-{latex[scale=3.0]}] (O2) -- (O4) node[right] {$\bm{N}_i$};
\draw[dotted,thick] (O2) -- (O5);
\draw[very thick,->] (O2) --  ($(O2)!1.2cm!(O6)$) node[above] {$\bm{T}_i$};
\draw[very thick,->] (O2) --  ($(O2)!1.2cm!(O7)$) node[above] {$\bm{N}_i$};
\draw
    pic[draw, ->,thick, angle eccentricity=1.4, angle radius=0.8cm]
    {angle=O5--O2--O3};
\draw
    pic["$\varphi_i$", angle eccentricity=1.4, angle radius=0.8cm]
    {angle=O6--O2--O3};
\draw
    pic["$\theta_i$", draw, ->,thick, angle eccentricity=1.4, angle radius=0.5cm]
    {angle=O0--O1--O2};
\draw
    pic["$-\frac{\pi}{2}$", draw, <-,thick, angle eccentricity=1.4, angle radius=0.6cm]
    {angle=O7--O2--O6};
\coordinate (a1) at (0,1);
\coordinate (a2) at (1.15,1.35);
\coordinate (a3) at (-0.35,2.15);
\draw[thick,->] (a1) -- ($(a1)!1.2cm!(a2)$) node[right] {$\bm{n}_i$};
\draw[thick,->] (a1) -- ($(a1)!1.2cm!(a3)$) node[right] {$\bm{t}_i$};
\coordinate (b1) at (-3,3.7);
\coordinate (b2) at (-2.8,4.7);
\coordinate (b3) at (-4,3.9);
\draw[thick,->] (b1) -- ($(b1)!1.2cm!(b2)$) node[right] {$\bm{n}_{i+1}$};
\draw[thick,->] (b1) -- ($(b1)!1.2cm!(b3)$) node[above] {$\bm{t}_{i+1}$};
\end{tikzpicture}
	    \caption{Definitions of the unit tangent and unit outward normal vectors at vertices.}
	    \label{fig:T-N_polygonal_curve}
\end{figure}
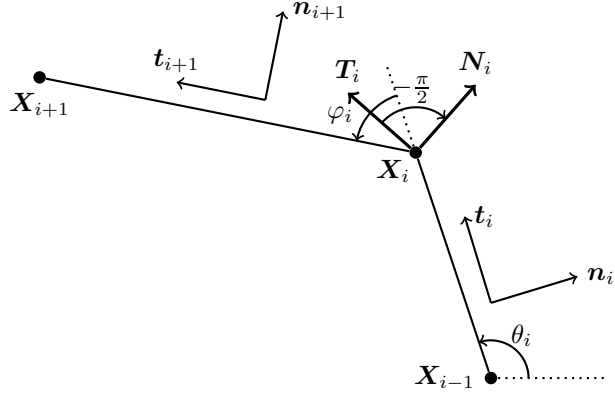
	
\begin{lem}
\label{lem:T_N_t_n}
We have
\begin{alignat*}{2}
	&\bm{t}_i^t=\cos(\varphi_i^t/2)\bm{T}_i^t+\sin(\varphi_i^t/2)\bm{N}_i^t,
	&\quad
	& \bm{t}_{i+1}^t=\cos(\varphi_i^t/2)\bm{T}_i^t-\sin(\varphi_i^t/2)\bm{N}_i^t,\\
	&\bm{n}_i^t=-\sin(\varphi_i^t/2)\bm{T}_i^t+\cos(\varphi_i^t/2)\bm{N}_i^t,
	& \quad
	&\bm{n}_{i+1}^t=\sin(\varphi_i^t/2)\bm{T}_i^t+\cos(\varphi_i^t/2)\bm{N}_i^t.
\end{alignat*}
\end{lem}

Below, we define a discrete curvature $\kappa_i^t$ and show how the length and enclosed area evolve.	
The  length $\mathcal{L}[\Gamma]$ of a polygonal curve $\Gamma$ is naturally defined by
\begin{equation}
	\mathcal{L}[\Gamma]
	=
	\sum_{i=1}^Nr_i.
\end{equation}
It follows from the semi-discrete evolution law \eqref{eq:polygonal_evolution_law} and Lemma \ref{lem:T_N_t_n}  that
\begin{align}
	\partial_t\mathcal{L}[\Gamma^t]
	&=
	\sum_{i=1}^N\partial_tr_i^t
	=
	\sum_{i=1}^N\partial_{\Gamma^t,i}\bm{X}^t\cdot\partial_t(\bm{X}_i^t-\bm{X}_{i-1}^t)\\
	&=
	\sum_{i=1}^N(\bm{t}_i^t-\bm{t}_{i+1}^t)\cdot(V_i^t\bm{N}_i^t+W_i^t\bm{T}_i^t)
	=
	2\sum_{i=1}^N\sin\frac{\varphi_i^t}{2}V_i^t.
\end{align}
If we define the relation between the normal velocities on edges and vertices by
\begin{align}
	V_i^t=\av_{\Gamma^t,i}v^t,
	\quad
	i=1,2,\ldots,N,
	\label{eq:relation_normal_velocity}
\end{align}
the time derivative of the  length can be expressed as	
\begin{equation}
	\partial_t\mathcal{L}[\Gamma^t]
	=
	2\sum_{i=1}^N\sin\frac{\varphi_i^t}{2}\av_{\Gamma^t,i}v^t
	=
	\sum_{i=1}^N\frac{\tan(\varphi_i^t/2)+\tan(\varphi_{i-1}^t/2)}{r_i^t}v_i^tr_i^t.
\end{equation}
Comparing this expression with the first relation in \eqref{eq:dtdL_dtdA_dtdG}, we define a discrete curvature $\kappa_i^t$ on the $i$th edge $\Gamma_i^t$ by
\begin{equation}
	\kappa_i^t
	:=
	\frac{\tan(\varphi_i^t/2)+\tan(\varphi_{i-1}^t/2)}{r_i^t},
	\quad
	i=1,2,\ldots,N.
	\label{eq:discrete_curvature}
\end{equation}
The next theorem, which is a discrete analogue of Theorem \ref{thm:dtdL_dtdA_dtdG}, characterizes the evolution of the  length and enclosed area in terms of the discrete curvature and normal velocity.
	
\begin{thm}
\label{thm:dtdL_dtdA_pmbp}
%Suppose that a polygonal curve $\Gamma^t$ evolves according to the polygonal evolution law \eqref{eq:polygonal_evolution_law}.
%Then, we have the following relations:
For the polygonal curve $\Gamma^t$ that evolves according to the semi-discrete evolution law \eqref{eq:polygonal_evolution_law}, it follows that
\begin{equation}
	\partial_t\mathcal{L}[\Gamma^t]=\sum_{i=1}^N\kappa_i^tv_i^tr_i^t,
	\quad
	\partial_t\mathcal{A}[\Omega^t]=\sum_{i=1}^Nv_i^tr_i^t+\mathrm{err}_{\mathcal{A}[\Omega^t]},
	\label{eq:dtdL_dtdA_pmbp}
\end{equation}
where $\Omega^t$ represents the polygonal region bounded by $\Gamma^t$, and the error term $\mathrm{err}_{\mathcal{A}[\Omega^t]}$ is given by
\begin{equation}
	\mathrm{err}_{\mathcal{A}[\Omega^t]}
	=
	\sum_{i=1}^N\left(W_i^t\sin\frac{\varphi_i^t}{2}-\frac{v_{i+1}^t-v_i^t}{2}\right)\frac{r_{i+1}^t-r_i^t}{2}.
\end{equation}
\end{thm}
	
\begin{proof}
The relation for the  length $\mathcal{L}[\Gamma^t]$ is an immediate consequence of the definition of the discrete curvature \eqref{eq:discrete_curvature}.
The relation for the enclosed area $\mathcal{A}[\Omega^t]$ first appeared in \cite[Proposition 4]{SY19} without proof; therefore, we give its proof.

The  enclosed area $\mathcal{A}[\Omega]$ can be expressed as
\begin{equation}
	\mathcal{A}[\Omega]
	=
	\int_\Omega\,\mathrm{d}\bm{x}
	=
	\frac{1}{2}\sum_{i=1}^N(\bm{X}_i\cdot\bm{n}_i)r_i
	=
	\frac{1}{2}\sum_{i=1}^NJ\bm{X}_{i-1}\cdot\bm{X}_i.
\end{equation}
Thus, using the semi-discrete evolution law \eqref{eq:polygonal_evolution_law}, Lemma \ref{lem:T_N_t_n} and the definition of $V_i^t$ \eqref{eq:relation_normal_velocity}, we have
\begin{align}
	\partial_t\mathcal{A}[\Omega^t]
	&=
	\frac{1}{2}\sum_{i=1}^N\left(J\partial_t\bm{X}_{i-1}^t\cdot\bm{X}_i^t+J\bm{X}_{i-1}^t\cdot\partial_t\bm{X}_i^t\right)\\
	&=
	\frac{1}{2}\sum_{i=1}^NJ(\bm{X}_{i-1}^t-\bm{X}_{i+1}^t)\cdot\partial_t\bm{X}_i^t\\
	&=
	\frac{1}{2}\sum_{i=1}^N(r_{i+1}^t\bm{n}_{i+1}^t+r_i^t\bm{n}_i^t)\cdot(V_i^t\bm{N}_i^t+W_i^t\bm{T}_i^t)\\
	&=
	\frac{1}{4}\sum_{i=1}^N(v_i^t+v_{i+1}^t)(r_i^t+r_{i+1}^t)+\sum_{i=1}^NW_i\sin\frac{\varphi_i^t}{2}\frac{r_{i+1}^t-r_i^t}{2}\\
	&=
	\sum_{i=1}^Nv_i^tr_i^t+\sum_{i=1}^N\left(W_i\sin\frac{\varphi_i^t}{2}-\frac{v_{i+1}^t-v_i^t}{2}\right)\frac{r_{i+1}^t-r_i^t}{2}.
	\label{par_A_Omega}
\end{align}
\end{proof}

It is hoped that the error term $\mathrm{err}_{\mathcal{A}[\Omega^t]}$ is as close as possible to $0$.
We will explain in the next subsection that the term can be controlled by manipulating the tangential velocity.
	
\begin{rem} %\memo{どこに書くのがよい？}
\label{rem:N}
Note that the tangent and normal vectors on vertices were defined by
$\bm{T}_i^t = \av_{\Gamma^t,i}\bm{t}^t$ and $\bm{N}_i^t = \av_{\Gamma^t,i}\bm{n}^t$, respectively.
The definitions are also understood in the following way.
The normal vector $\bm{N}_i^t$ is parallel to $\bm{t}_i^t-\bm{t}_{i+1}^t$, and the tangent vector $\bm{T}_i^t$ is defined so that it is perpendicular to $\bm{N}_i^t$.
\end{rem}
	
\subsection{Tangential velocity for semi-discrete moving boundary problem}
\label{subsec:tangential_velocity_pmbp}

As explained in Section~\ref{sec:introduction}, there exist several approaches to setting appropriate discrete tangential velocities.
Among them, we here employ the asymptotic uniform distribution method~\cite{SY13}, which defines discrete tangential velocities so that the arrangement of the vertices is made uniform as $t\to\infty$. 

Given real numbers $\eta_i$ ($i=1,\dots,N$) satisfying $\sum_{i=1}^N\eta_i=0$ and given function $f^t(N)$ satisfying $\lim_{t\to \infty}f^t(N)=\infty$, we require that the curve $\Gamma^t$ satisfies
\begin{equation}
	r_i^t-\frac{\mathcal{L}[\Gamma^t]}{N}
	=
	\eta_i\exp(-f^t(N)),
	\quad
	i=1,2,\ldots,N.
	\label{eq:criterion_UDM}
\end{equation}
Differentiating \eqref{eq:criterion_UDM} gives
\begin{equation}
	\partial_tr_i^t-\frac{\partial_t\mathcal{L}[\Gamma^t]}{N}
	=
	\left(\frac{\mathcal{L}[\Gamma^t]}{N}-r_i^t\right)\omega^t(N),
	\label{eq:UDM_1}
\end{equation}
where $\omega^t(N)=\partial_tf^t(N)$, and we call $\omega^t(N)$ a relaxation parameter.
On the other hand, under the semi-discrete evolution law  \eqref{eq:polygonal_evolution_law} the time derivative of the length of the $i$th edge is calculated to be
\begin{align}
	\partial_tr_i^t
	&=
	\partial_{\Gamma^t,i}\bm{X}^t\cdot\partial_t(\bm{X}_i^t-\bm{X}_{i-1}^t)\\
	&=
	V_i^t\sin\frac{\varphi_i^t}{2}+V_{i-1}^t\sin\frac{\varphi_{i-1}^t}{2}+W_i^t\cos\frac{\varphi_i^t}{2}-W_{i-1}^t\cos\frac{\varphi_{i-1}^t}{2}.
	\label{eq:UDM_2}
\end{align}
In order for \eqref{eq:UDM_1} and \eqref{eq:UDM_2} to be compatible,
\begin{multline}
	W_i^t\cos\frac{\varphi_i^t}{2}-W_{i-1}^t\cos\frac{\varphi_{i-1}^t}{2}\\
	=
	-V_i^t\sin\frac{\varphi_i^t}{2}-V_{i-1}^t\sin\frac{\varphi_{i-1}^t}{2}+\frac{\partial_t\mathcal{L}[\Gamma^t]}{N}+\left(\frac{\mathcal{L}[\Gamma^t]}{N}-r_i^t\right)\omega^t(N)
	\label{cond:Wit}
\end{multline}
must be satisfied for $i=2,3,\ldots,N$.
Here, because summing up the condition \eqref{cond:Wit} for $i=2,3,\ldots,N$ gives the same condition for $i=1$, it is sufficient to consider $i=2,3,\ldots,N$.
The asymptotic uniform distribution method defines the tangential velocities so that they satisfy  the condition \eqref{cond:Wit}, but we note that the tangential velocities $\{W_i^t\}_{i=1}^N$ are still underdetermined under the condition \eqref{cond:Wit} for $i=2,3,\ldots,N$.
To determine the tangential velocities $\{W_i^t\}_{i=1}^N$ uniquely, we need to impose an additional condition.
For example, requiring $ \sum_{i=1}^NW_i^t(r_i^t+r_{i+1}^t)/2=0$ gives the tangential velocities $\{W_i^t\}_{i=1}^N$ explicitly:
\begin{align}
	W_1^t
	&=
	-\frac{\sum_{i=2}^N\Psi_i^t(r_i^t+r_{i+1}^t)/(2\cos(\varphi_i^t/2))}{\cos(\varphi_1^t/2)\sum_{i=1}^N(r_i^t+r_{i+1}^t)/(2\cos(\varphi_i^t/2))}, 
	\\
	W_i^t
	&=
	\frac{\Psi_i^t+W_1^t\cos(\varphi_1^t/2)}{\cos(\varphi_i^t/2)}, \quad i=2,3,\ldots,N,
\end{align}
where
\begin{align}
	\Psi_i^t
	&
	=\sum_{l=2}^i\psi_i^t,\\
	\psi_i^t
	&=
	-V_i^t\sin\frac{\varphi_i^t}{2}-V_{i-1}^t\sin\frac{\varphi_{i-1}^t}{2}+\frac{\partial_t\mathcal{L}[\Gamma^t]}{N}+\left(\frac{\mathcal{L}[\Gamma^t]}{N}-r_i^t\right)\omega^t(N)
\end{align}
for $i=2,3,\ldots,N$.
	
The asymptotic uniform distribution method makes the positions of the vertices uniform as $t\to \infty$.
In particular, if the position of the initial vertices $\{\bm{X}_i^0\}_{i=1}^N$ is uniform, that is, $r_i^0\equiv\mathcal{L}[\Gamma^0]/N$ holds for all $i$, then the method keeps the uniformity along the time evolution:
$r_i^t\equiv\mathcal{L}[\Gamma^t]/N$ holds for all $i$ and all $t$.
The following is an immediate consequence of the above discussion, and can be seen as a  discrete analogue of Theorem \ref{thm:dtdL_dtdA_dtdG}.
	
\begin{cor}
\label{cor:dtdL_dtdA_pbmp_audm}
Suppose that the initial vertices are allocated uniformly, i.e. $r_i^0=\text{const.}$ and that the tangential velocities $\{W_i^t\}_{i=1}^N$ are set such that they satisfy \eqref{cond:Wit}.
Then, for the polygonal curve $\Gamma^t$ that evolves according to the semi-discrete evolution law \eqref{eq:polygonal_evolution_law}, it follows that
%	    In addition to the hypothesis in Theorem \ref{thm:dtdL_dtdA_pmbp}, suppose that
%	    the initial vertices are allocated uniformly, i.e. $r_i^0=\text{const.}$ and that the tangential velocities $\{W_i^t\}_{i=1}^N$ are set so that they satisfy \eqref{cond:Wit}.
%	    Then, we have
\begin{equation}
	\partial_t\mathcal{L}[\Gamma^t]
	=
	\sum_{i=1}^N\kappa_i^tv_i^tr_i^t,
	\quad
	\partial_t\mathcal{A}[\Omega^t]
	=
	\sum_{i=1}^Nv_i^tr_i^t.
\end{equation}
\end{cor}
	
\subsection{Examples}
We consider discretizing the mean curvature flow, area-preserving mean curvature flow and Hele-Shaw flow.
In this subsection, we suppose that the initial vertices are arranged uniformly and that the tangential velocities are set based on the asymptotic uniform distribution method.
Note that for each problem, specifying the normal velocity on edges determines the dynamics of the semi-discrete flow.
	
\begin{ex}[semi-discrete polygonal mean curvature flow]
\label{ex:mean_curvature_flow_pmbp}
The normal velocity of the semi-discrete polygonal mean curvature flow is given by
\begin{equation}
	v_i^t=-\kappa_i^t, \quad i=1,2,\ldots,N,\ t>0.
\end{equation}
We then have the curve-shortening property:
\begin{equation}
	\partial_t\mathcal{L}[\Gamma^t]
	=
	-\sum_{i=1}^N(\kappa_i^t)^2r_i^t
	<
	0.
\end{equation}
Moreover, we have
\begin{equation}
	\partial_t\mathcal{A}[\Omega^t]
	=
	-2\sum_{i=1}^N\tan\frac{\varphi_i^t}{2}.
	\label{der_A_semi}
\end{equation}
We note that $\sum_{i=1}^N\tan(\varphi_i^t/2)-\pi=\mathrm{O}(N^{-2})$ as $N\to\infty$ since $\tan(\varphi_i^t/2)-\varphi_i^t/2=\mathrm{O}(N^{-3})$, and thus $\partial_t\mathcal{A}[\Omega^t] + 2\pi = \mathrm{O}(N^{-2})$.
Therefore,  \eqref{der_A_semi} can be regarded as a discrete counterpart of \eqref{eq:MCF_area_speed}.
\end{ex}
	
\begin{ex}[semi-discrete polygonal area-preserving mean curvature flow]
\label{ex:area-preserving_mean_curvature_flow_pmbp}
The normal velocity of the semi-discrete polygonal area-preserving mean curvature flow is given by
\begin{equation}
	v_i^t=-\kappa_i^t+\langle\kappa^t\rangle_{\Gamma^t},
	\quad
	i=1,2,\ldots,N,\ t>0,
\end{equation}
where $\langle\mathsf{f}\rangle_\Gamma$ denotes the average of $\{\mathsf{f}_i\}_{i=1}^N$ along the curve $\Gamma$:
\begin{equation}
	\langle\mathsf{f}\rangle_\Gamma
	=
	\frac{1}{\mathcal{L}[\Gamma]}\sum_{i=1}^N\mathsf{f}_ir_i.
\end{equation}
We then have the curve-shortening property: 
\begin{align}
	\partial_t\mathcal{L}[\Gamma^t]
	&=
	\sum_{i=1}^N\kappa_i^t\left(-\kappa_i^t+\langle\kappa^t\rangle_{\Gamma^t}\right)r_i^t
	=
	-\sum_{i=1}^N(\kappa_i^t)^2r_i^t+\frac{1}{\mathcal{L}[\Gamma^t]}\left(\sum_{i=1}^N\kappa_i^tr_i^t\right)^2\\
	&=
	\frac{1}{\mathcal{L}[\Gamma^t]}\left[-\sum_{i=1}^Nr_i^t\sum_{i=1}^N(\kappa_i^t)^2r_i^t+\left(\sum_{i=1}^N\kappa_i^tr_i^t\right)^2\right]
	\le
	0
\end{align}
and the area-preserving property:
\begin{equation}
	\partial_t\mathcal{A}[\Omega^t]
	=
	\sum_{i=1}^N\left(-\kappa_i^t+\langle\kappa^t\rangle_{\Gamma^t}\right)r_i^t
	=
	0.
\end{equation}
\end{ex}
	
\begin{ex}[semi-discrete polygonal Hele-Shaw flow]
\label{ex:Hele-Shaw_flow_pmbp}
The normal velocity of the semi-discrete polygonal Hele-Shaw flow is given by
\begin{equation}
	v_i^t=-\nabla p^t\cdot\bm{n}_i^t,
	\quad
	i=1,2,\ldots,N,\ t>0,
\end{equation}
where $p^t$ is the solution to the Laplace equation: 
\begin{equation}
	\begin{dcases*}
		\triangle p^t=0&in $\Omega^t$,\\
	    p^t=\sigma\kappa_i^t&on $\Gamma_i^t$, $i=1,2,\ldots,N$,\\
		\nabla p^t\cdot\bm{n}_i^t\equiv\text{cosnt.}&on $\Gamma_i^t$, $i=1,2,\ldots,N$.
	\end{dcases*}
\end{equation}
We then have the curve-shortening property:
\begin{align}
	\partial_t\mathcal{L}[\Gamma^t]
	&=
	-\frac{1}{\sigma}\sum_{i=1}^Np^t|_{\Gamma_i^t}\nabla p^t|_{\Gamma_i^t}\cdot\bm{n}_i^tr_i^t
	=
	-\frac{1}{\sigma}\int_{\Gamma^t}p^t\nabla p^t\cdot\bm{n}^t\,\mathrm{d}s\\
	&=
	-\frac{1}{\sigma}\int_{\Omega^t}\nabla\cdot(p^t\nabla p^t)\,\mathrm{d}\bm{x}
	=
	-\frac{1}{\sigma}\int_{\Omega^t}|\nabla p^t|^2\,\mathrm{d}\bm{x}
	\le
	0
\end{align}
and the area-preserving property:
\begin{equation}
	\partial_t\mathcal{A}[\Omega^t]
	=
	-\sum_{i=1}^N\nabla p^t|_{\Gamma_i^t}\cdot\bm{n}_i^tr_i^t
	=
	-\int_{\Gamma^t}\nabla p^t\cdot\bm{n}^t\,\mathrm{d}s
	=
	-\int_{\Omega^t}\nabla\cdot\nabla p^t\,\mathrm{d}\bm{x}
	=
	0,
\end{equation}
where $\bm{n}^t:=\sum_{i=1}^N\bm{n}_i^t\chi_{\Gamma_i^t}$, and $\chi_{\Gamma_i}$ denotes the characteristic function of $\Gamma_i$.
\end{ex}
	
\section{Fully discrete polygonal moving boundary problems}
\label{sec:time-discretized_polygonal_moving_boundary_problem}

In this section, we derive a fully discrete evolution law that inherits the curve-shortening property for the moving boundary problem \eqref{eq:evolution_law}.

\subsection{Fully discrete evolution law}
	
Let us consider a fully discrete evolution law of the form 
\begin{multline}
	\frac{\bm{X}_i^{n+1}-\bm{X}_i^{n}}{\Delta t^{n}}
	=
	V_i(\Gamma^{n+1},\Gamma^{n})\bm{N}_i(\Gamma^{n+1},\Gamma^{n})+W_i(\Gamma^{n+1},\Gamma^{n})\bm{T}_i(\Gamma^{n+1},\Gamma^{n}),\\
	i=1,2,\ldots,N,\ n=0,1,\ldots,
	\label{eq:discretized_polygonal_evolution_law}
\end{multline}
where $\Delta t^{n}$ is the time step size and $\Gamma^n$ denotes the polygonal curve at the $n$th time step $t=t_n$.
Note that the vectors $\bm{T}_i(\Gamma^{n+1},\Gamma^{n})$ and $\bm{N}_i(\Gamma^{n+1},\Gamma^{n})$ and velocities $V_i(\Gamma^{n+1},\Gamma^{n})$ and $W_i(\Gamma^{n+1},\Gamma^{n})$ depend on both $\Gamma^{n}$ and $\Gamma^{n+1}$, and their definitions will be given in the subsequent subsections.

\subsection{Definitions of the unit tangent vector, unit outward normal vector and curvature}
\label{subsec:FDvectors}
To simplify the notation, we simply omit the superscript $n+1$ and use the hat instead of the superscript $n$. The time step size is simply denoted by $\Delta t$.
The evolution law \eqref	{eq:discretized_polygonal_evolution_law} is then expressed as
\begin{equation}
	\frac{\bm{X}_i-\hat{\bm{X}}_i}{\Delta t}
	=
	V_i(\Gamma,\hat{\Gamma})\bm{N}_i(\Gamma,\hat{\Gamma})+W_i(\Gamma,\hat{\Gamma})\bm{T}_i(\Gamma,\hat{\Gamma}),
	\quad
	i=1,2,\ldots,N,
\end{equation}
where
\begin{equation}
	\Gamma
	=
	\bigcup_{i=1}^N\Gamma_i,\ 
	\Gamma_i=(\bm{X}_{i-1},\bm{X}_i),
	\quad
	\hat{\Gamma}
	=
	\bigcup_{i=1}^N\hat{\Gamma}_i,\ 
	\hat{\Gamma}_i=(\hat{\bm{X}}_{i-1},\hat{\bm{X}}_i).
\end{equation}
In this subsection, we define the unit tangent vector $\bm{T}_i(\Gamma,\hat{\Gamma})$, unit outward normal vector $\bm{N}_i(\Gamma,\hat{\Gamma})$, normal velocity $V_i(\Gamma,\hat{\Gamma})$ and curvature.
	
\subsubsection{Unit tangent vector and unit outward normal vector}
The length of the $i$th edge of the polygonal curve $\Gamma$ is denoted by $r_i$, i.e.
\begin{equation}
	r_i=|\bm{X}_i-\bm{X}_{i-1}|,
	\quad
	\hat{r}_i=|\hat{\bm{X}}_i-\hat{\bm{X}}_{i-1}|,
	\quad
	i=1,2,\ldots,N.
\end{equation}
The unit tangent vectors $\bm{t}_i$ on $\Gamma_i$ and $\hat{\bm{t}}_i$ on $\hat{\Gamma}_i$ are defined straightforwardly by
\begin{equation}
	\bm{t}_i=\partial_{\Gamma,i}\bm{X},
	\quad
	\hat{\bm{t}}_i=\partial_{\hat{\Gamma},i}\hat{\bm{X}},
	\quad
	i=1,2,\ldots,N.
\end{equation}
We then have 
\begin{align}
	\frac{\mathcal{L}[\Gamma]-\mathcal{L}[\hat{\Gamma}]}{\Delta t}
	&=
	\sum_{i=1}^N\frac{r_i-\hat{r}_i}{\Delta t}
	=
	\sum_{i=1}^N\frac{r_i\bm{t}_i+\hat{r}_i\hat{\bm{t}}_i}{r_i+\hat{r}_i}\cdot\frac{(\bm{X}_i-\bm{X}_{i-1})-(\hat{\bm{X}}_i-\hat{\bm{X}}_{i-1})}{\Delta t}\\
	&=
	-\sum_{i=1}^N  \left(\frac{r_{i+1}\bm{t}_{i+1}+\hat{r}_{i+1}\hat{\bm{t}}_{i+1}}{r_{i+1}+\hat{r}_{i+1}}-\frac{r_i\bm{t}_i+\hat{r}_i\hat{\bm{t}}_i}{r_i+\hat{r}_i}\right)\cdot\frac{\bm{X}_i-\hat{\bm{X}}_i}{\Delta t}\\
	&=
	-\sum_{i=1}^N  \left(\partial_{\Gamma,\hat{\Gamma},i+1}\overline{\bm{X}}-\partial_{\Gamma,\hat{\Gamma},i}\overline{\bm{X}}\right)\cdot\frac{\bm{X}_i-\hat{\bm{X}}_i}{\Delta t}, 
	\label{fd:length_calc}
\end{align}
where $\overline{\bm{X}}_i:=(\bm{X}_i+\hat{\bm{X}}_i)/2$, and $\partial_{\Gamma,\hat{\Gamma},i}$ denotes the difference operator associated with two curves $\Gamma$ and $\hat{\Gamma}$:
\begin{equation}
	\partial_{\Gamma,\hat{\Gamma},i}\mathsf{F}
	=
	\frac{\mathsf{F}_i-\mathsf{F}_{i-1}}{(r_i+\hat{r}_i)/2},
	\quad
	i=1,2,\ldots,N.
\end{equation}
This operator operates on a function on edges.

Below, we define $\overline{\bm{T}}_i=\bm{T}_i(\Gamma,\hat{\Gamma})$ and $\overline{\bm{N}}_i=\bm{N}_i(\Gamma,\hat{\Gamma})$. 
We note that $\partial_{\Gamma,\hat{\Gamma},i}\overline{\bm{X}}$, which can be regarded as a tangent vector on the $i$th edge, is not necessarily a unit vector; thus it is not straightforward to define $\overline{\bm{T}}_i$ and $\overline{\bm{N}}_i$.
In contrast to the time-continuous cases discussed in Section~\ref{sec:polygonal_moving_boundary_problem}, we start by defining $\overline{\bm{N}}_i$.
Recall that in Section~\ref{sec:polygonal_moving_boundary_problem} the normal vector $\bm{N}_i^t$ can be understood in such a way that it is parallel to $\bm{t}_i^t-\bm{t}_{i+1}^t$ (Remark~\ref{rem:N}).
Motivated by this interpretation, we define $\overline{\bm{N}}_i$ by requiring it to be parallel to $\partial_{\Gamma,\hat{\Gamma},i+1}\overline{\bm{X}}-\partial_{\Gamma,\hat{\Gamma},i}\overline{\bm{X}}$:
\begin{equation}
	\overline{\bm{N}}_i = \bm{N}_i (\Gamma,\hat{\Gamma})
	=
	\begin{dcases*}
		-\sgn\overline{\varphi}_i\frac{\partial_{\Gamma,\hat{\Gamma},i+1}\overline{\bm{X}}-\partial_{\Gamma,\hat{\Gamma},i}\overline{\bm{X}}}{|\partial_{\Gamma,\hat{\Gamma},i+1}\overline{\bm{X}}-\partial_{\Gamma,\hat{\Gamma},i}\overline{\bm{X}}|}&if $\overline{\varphi}_i\neq0$,\\
		-J\overline{\bm{t}}_i&if $\overline{\varphi}_i=0$,
	\end{dcases*}
\end{equation}
where $\overline{\varphi}_i$ denotes the signed angle between two adjacent edges $\overline{\Gamma}_i$ and $\overline{\Gamma}_{i+1}$.
The unit tangent vector $\overline{\bm{T}}_i = \bm{T}_i (\Gamma,\hat{\Gamma})$ is defined such that it is perpendicular to $\overline{\bm{N}}_i$: $\overline{\bm{T}}_i=J\overline{\bm{N}}_i$.

Using $\overline{\bm{N}}_i$ and $\overline{\bm{T}}_i$, we can represent $\partial_{\Gamma,\hat{\Gamma},i+1}\overline{\bm{X}}-\partial_{\Gamma,\hat{\Gamma},i}\overline{\bm{X}}$ as follows:
\begin{equation}
	\partial_{\Gamma,\hat{\Gamma},i+1}\overline{\bm{X}}-\partial_{\Gamma,\hat{\Gamma},i}\overline{\bm{X}}=\overline{F}_i\overline{\bm{N}}_i+\overline{G}_i\overline{\bm{T}}_i,
\end{equation}
where
\begin{align}
	\overline{F}_i
	&=
	\begin{dcases*}
		-\frac{|\partial_{\Gamma,\hat{\Gamma},i+1}\overline{\bm{X}}-\partial_{\Gamma,\hat{\Gamma},i}\overline{\bm{X}}|}{\sgn\overline{\varphi}_i}&if $\overline{\varphi}_i\neq0$,\\
		0&if $\overline{\varphi}_i=0$,
	\end{dcases*}\\
	\overline{G}_i
	&=
	\begin{dcases*}
		0&if $\overline{\varphi}_i\neq0\lor\left(\overline{\varphi}_i=0\land|\partial_{\Gamma,\hat{\Gamma},i+1}\overline{\bm{X}}|=|\partial_{\Gamma,\hat{\Gamma},i}\overline{\bm{X}}|\right)$,\\
		|\partial_{\Gamma,\hat{\Gamma},i+1}\overline{\bm{X}}-\partial_{\Gamma,\hat{\Gamma},i}\overline{\bm{X}}|&if $\overline{\varphi}_i=0\land|\partial_{\Gamma,\hat{\Gamma},i+1}\overline{\bm{X}}|>|\partial_{\Gamma,\hat{\Gamma},i}\overline{\bm{X}}|$,\\
		-|\partial_{\Gamma,\hat{\Gamma},i+1}\overline{\bm{X}}-\partial_{\Gamma,\hat{\Gamma},i}\overline{\bm{X}}|&if $\overline{\varphi}_i=0\land|\partial_{\Gamma,\hat{\Gamma},i+1}\overline{\bm{X}}|<|\partial_{\Gamma,\hat{\Gamma},i}\overline{\bm{X}}|$.
	\end{dcases*}
\end{align}
Note that the case $\overline{\varphi}_i=0$ seldom occurs.
By using $\overline{F}_i$ and $\overline{G}_i$, the calculation of $(\mathcal{L}[\Gamma]-\mathcal{L}[\hat{\Gamma}])/\Delta t$ in \eqref{fd:length_calc} further proceeds as follows:
\begin{align}
	\frac{\mathcal{L}[\Gamma]-\mathcal{L}[\hat{\Gamma}]}{\Delta t}
	&=
	-\sum_{i=1}^N(\overline{F}_i\overline{\bm{N}}_i+\overline{G}_i\overline{\bm{T}}_i)\cdot(\overline{V}_i\overline{\bm{N}}_i+\overline{W}_i\overline{\bm{T}}_i)\\
	&=
	-\sum_{i=1}^N\overline{F}_i\overline{V}_i-\sum_{i=1}^N\overline{G}_i\overline{W}_i.
	\label{eq:time_difference_length_1}
\end{align}
	
\subsubsection{Normal velocity and curvature}
Let us consider the normal velocities $\{\overline{V}_i\}_{i=1}^N$ on vertices and the curvatures $\{\overline{\kappa}_i\}_{i=1}^N$ on edges.
We start by representing $\overline{\bm{N}}_i$ as a linear combination of $\overline{\bm{n}}_i$ and $\overline{\bm{n}}_{i+1}$, where $\overline{\bm{n}}_i = -J \overline{\bm{t}}_i$.
If $\overline{\varphi}_i\neq0$, 
$\overline{\bm{n}}_i$ and $\overline{\bm{n}}_{i+1}$ are linearly independent.
Thus the representation $\overline{\bm{N}}_i=\overline{\alpha}_i\overline{\bm{n}}_i+\overline{\beta}_i\overline{\bm{n}}_{i+1}$ is uniquely determined, where $\overline{\alpha}_i$ and $\overline{\beta}_i$ are
the unique solutions to the linear system
\begin{equation}
	\begin{dcases*}
		\overline{\bm{N}}_i\cdot\overline{\bm{n}}_i
	    =
		\overline{\alpha}_i+(\overline{\bm{n}}_i\cdot\overline{\bm{n}}_{i+1})\overline{\beta}_i,\\
		\overline{\bm{N}}_i\cdot\overline{\bm{n}}_{i+1}
		=
	    (\overline{\bm{n}}_i\cdot\overline{\bm{n}}_{i+1})\overline{\alpha}_i+\overline{\beta}_i.
	\end{dcases*}
\end{equation}
If $\overline{\varphi}_i=0$, then $\overline{\bm{n}}_i=\overline{\bm{n}}_{i+1}$ holds, and there are infinitely many possible candidates for the expression.
In order for the two cases $\overline{\varphi}_i\neq 0$ and $\overline{\varphi}_i= 0$ to be consistent,
we employ the following one:
\begin{equation}
	\overline{\bm{N}}_i=\overline{\alpha}_i\overline{\bm{n}}_i+\overline{\beta}_i\overline{\bm{n}}_{i+1},
	\quad
	\overline{\alpha}_i=\frac{\overline{r}_i}{\overline{r}_i+\overline{r}_{i+1}},
	\quad
	\overline{\beta}_i=\frac{\overline{r}_{i+1}}{\overline{r}_i+\overline{r}_{i+1}},
\end{equation}
where $\overline{r}_i=|\overline{\bm{X}}_i-\overline{\bm{X}}_{i-1}|$. 
Using the above coefficients $\overline{\alpha}_i$ and $\overline{\beta}_i$, we define the normal velocity on the $i$th edge by
\begin{equation}
	\overline{V}_i=\overline{\alpha}_i\overline{v}_i+\overline{\beta}_i\overline{v}_{i+1},
	\quad
	i=1,2,\ldots,N.
	\label{eq:relation_V_v}
\end{equation}
Substituting this expression into \eqref{eq:time_difference_length_1}, we obtain 
\begin{align}
	\frac{\mathcal{L}[\Gamma]-\mathcal{L}[\hat{\Gamma}]}{\Delta t}
	&=
	-\sum_{i=1}^N\overline{F}_i(\overline{\alpha}_i\overline{v}_i+\overline{\beta}_i\overline{v}_{i+1})-\sum_{i=1}^N\overline{G}_i\overline{W}_i\\
	&=
	-\sum_{i=1}^N(\overline{F}_i\overline{\alpha}_i+\overline{F}_{i-1}\overline{\beta}_{i-1})\overline{v}_i-\sum_{i=1}^N\overline{G}_i\overline{W}_i\\
	&=
	-\sum_{i=1}^N\left(\frac{\overline{F}_i\overline{\alpha}_i+\overline{F}_{i-1}\overline{\beta}_{i-1}}{\overline{r}_i}\right)\overline{v}_i\overline{r}_i-\sum_{i=1}^N\overline{G}_i\overline{W}_i.
\end{align}
If the term $-\sum_{i=1}^N\overline{G}_i\overline{W}_i$ is negligible, the above relation can be regarded as a discrete analogue of the first relation in \eqref{eq:dtdL_dtdA_dtdG}.  Defining the discrete curvature $\overline{\kappa}_i$ on $\overline{\Gamma}_i$ by
\begin{equation}
	\overline{\kappa}_i=-\frac{\overline{F}_i\overline{\alpha}_i+\overline{F}_{i-1}\overline{\beta}_{i-1}}{\overline{r}_i},
	\quad
	i=1,2,\ldots,N,
\end{equation}
we have
\begin{equation}
	\frac{\mathcal{L}[\Gamma]-\mathcal{L}[\hat{\Gamma}]}{\Delta t}
	=
	\sum_{i=1}^N\overline{\kappa}_i\overline{v}_i\overline{r}_i-\sum_{i=1}^N\overline{G}_i\overline{W}_i.
	\label{eq:difference_length_with_error}
\end{equation}
Note that as the limit $\Delta t \downarrow0$, we obtain the definitions of the unit tangent vector $\bm{T}_i$, unit outward normal vector $\bm{N}_i$, and discrete curvature $\kappa_i$ which were defined in Section~\ref{sec:polygonal_moving_boundary_problem}.

\subsection{Tangential velocity}
\label{subsec:tangential_velocity}

We need to define tangential velocities to perform long-time stable numerical computation.
The following discussion is based on the asymptotic uniform distribution method, which was reviewed in Section~\ref{subsec:tangential_velocity_pmbp}.
 
We discretize the requirement \eqref{eq:UDM_1} in time to obtain the following system:
\begin{equation}
	\frac{r_i-\hat{r}_i}{\Delta t}-\frac{1}{N}\frac{\mathcal{L}[\Gamma]-\mathcal{L}[\hat{\Gamma}]}{\Delta t}
	=
	\left(\frac{\mathcal{L}[\Gamma]}{N}-r_i\right)\omega,
	\quad
	i=1,2,\ldots,N.
	\label{fd:r1}
\end{equation}
On the other hand, under the fully discrete evolution law \eqref{eq:discretized_polygonal_evolution_law}, the change in the length of each edge is calculated to be
\begin{align}
	\frac{r_i-\hat{r}_i}{\Delta t}
	=
	\frac{\overline{r}_i}{(r_i+\hat{r}_i)/2}
	\left(
	\begin{array}{l}
		(\overline{\bm{t}}_i\cdot\overline{\bm{N}}_i)\overline{V}_i-(\overline{\bm{t}}_i\cdot\overline{\bm{N}}_{i-1})\overline{V}_{i-1}\\
		\quad
		+(\overline{\bm{t}}_i\cdot\overline{\bm{T}}_i)\overline{W}_i-(\overline{\bm{t}}_i\cdot\overline{\bm{T}}_{i-1})\overline{W}_{i-1}
	\end{array}
	\right).
	\label{fd:r2}
\end{align}
In order for \eqref{fd:r1} and \eqref{fd:r2} to be compatible,
\begin{multline}
	\frac{\overline{r}_i(\overline{\bm{t}}_i\cdot\overline{\bm{T}}_i)}{(r_i+\hat{r}_i)/2}\overline{W}_i-\frac{\overline{r}_i(\overline{\bm{t}}_i\cdot\overline{\bm{T}}_{i-1})}{(r_i+\hat{r}_i)/2}\overline{W}_{i-1}\\
	=
	-\frac{\overline{r}_i(\overline{\bm{t}}_i\cdot\overline{\bm{N}}_i)}{(r_i+\hat{r}_i)/2}\overline{V}_i+\frac{\overline{r}_i(\overline{\bm{t}}_i\cdot\overline{\bm{N}}_{i-1})}{(r_i+\hat{r}_i)/2}\overline{V}_{i-1}+\frac{1}{N}\frac{\mathcal{L}[\Gamma]-\mathcal{L}[\hat{\Gamma}]}{\Delta t}+\left(\frac{\mathcal{L}[\Gamma]}{N}-r_i\right)\omega
	\label{eq:UDM}
\end{multline}
must be satisfied for $i=2,3,\ldots,N$.
    
Since $\{\overline{W}_i\}_{i=1}^N$ are still underdetermined, we need to add an additional condition.
In view of \eqref{eq:difference_length_with_error}, the zero weighted average condition
\begin{equation}
	\sum_{i=1}^N\overline{G}_i\overline{W}_i=0
	\label{eq:zero_weighted_average_condition}
\end{equation}
may be a candidate because this condition eliminates the second term of \eqref{eq:difference_length_with_error}.
However, this condition makes sense only when there exists at least one index $i$ such that $\overline{G}_i\neq0$, which merely happens.
We instead consider the simple zero average condition:
\begin{equation}
	\sum_{i=1}^N\overline{W}_i\overline{r}_i^*=0,
	\quad
	\overline{r}_i^*=\frac{\overline{r}_i+\overline{r}_{i+1}}{2}.
	\label{eq:zero_average_condition}
\end{equation}
We shall employ this condition in Section~\ref{sec:numerical_experiments}.

\begin{rem}
\label{rem:tangential_velocity}
The relation \eqref{fd:r1} can be rewritten as
\begin{equation}
	r_i-\frac{\mathcal{L}[\Gamma]}{N}
	=
	\frac{1}{1+\Delta t\omega}\left(\hat{r}_i-\frac{\mathcal{L}[\hat{\Gamma}]}{N}\right).
	\label{rel:omega}
\end{equation}
This implies that if the distribution of the initial vertices is uniform ($r_i^0\equiv \mathcal{L}[\Gamma^0]/N$ for all $i$), then the asymptotic uniform distribution method keeps the distribution uniform.
Even if $\hat{r}_i\neq\mathcal{L}[\hat{\Gamma}]/N$, the distribution tends to be uniform as the numerical integration proceeds (i.e. $r_i^n \to \mathcal{L}[\Gamma^n]/N$ as $n\to \infty$).
The relation \eqref{rel:omega} also indicates that the relaxation parameter should be set sufficiently large so that the convergence is quick.
\end{rem}
    
The following is an immediate consequence of the above discussion, which can be seen as a fully discrete counterpart of Theorem~\ref{thm:dtdL_dtdA_dtdG}.
 
\begin{thm}
\label{thm:discrete_dtdL}
Suppose that the initial vertices are allocated uniformly, i.e. $r_i^0=\text{const.}$ and that the tangential velocities $\{W_i^n\}_{i=1}^N$ are set such that they satisfy \eqref{eq:UDM}.
Then, for the polygonal curve $\Gamma^n$ that evolves according to the fully discrete evolution law \eqref{eq:discretized_polygonal_evolution_law}, it follows that
%Let $\Gamma^n=\bigcup_{i=1}^N\Gamma_i^n$, $\Gamma_i^n=(\bm{X}_{i-1}^n,\bm{X}_i^n)$, denote a polygonal Jordan curve at $n$th time step and suppose that its time evolution is described by the fully-discrete polygonal evolution law \eqref{eq:discretized_polygonal_evolution_law}, in which $\Delta t=\Delta t^{n}$,  $\Gamma=\Gamma^{n+1}$ and $\hat{\Gamma}=\Gamma^{n}$, that is, $\bm{X}_i=\bm{X}_i^{n+1}$ and $\hat{\bm{X}}_i=\bm{X}_i^{n}$.
%        Moreover, we assume the relations on normal velocities \eqref{eq:relation_V_v} and compute the tangential velocities by the asymptotic uniform distribution method \eqref{eq:UDM}
%        % and the zero weighted average condition \eqref{eq:zero_weighted_average_condition} or 
%        with the zero average condition \eqref{eq:zero_average_condition}.
%        We then have
\begin{equation}
 	\frac{\mathcal{L}[\Gamma^{n+1}]-\mathcal{L}[\Gamma^{n}]}{\Delta t^{n}}
	=
	\sum_{i=1}^N\kappa_i^{n+1/2}v_i^{n+1/2}r_i^{n+1/2}
\end{equation}
as long as $\overline{G}_i=0$ for all $i$, 
        % Namely, the discrete counterpart of the time derivative of the length functional holds, 
        where we denote $\overline{\kappa}_i$, $\overline{v}_i$, and $\overline{r}_i$ by $\kappa_i^{n+1/2}$, $v_i^{n+1/2}$, and $r_i^{n+1/2}$, respectively.
    \end{thm}
    
\subsection{Examples}
We apply the proposed discretization method to the mean curvature flow, area-preserving mean curvature flow and Hele-Shaw flow.
Note that for each problem, specifying the averaged normal velocity on edges determines the dynamics of the fully discrete flow.

\begin{ex}[fully discrete polygonal mean curvature flow]
\label{ex:time-discretized_mean_curvature_flow_pmbp}
The normal velocity of the fully  discrete polygonal mean curvature flow is given by
\begin{equation}
	v_i^{n+1/2}=-\kappa_i^{n+1/2},
	\quad
	i=1,2,\ldots,N,\ n=0,1,\ldots
\end{equation}
Then, we have the curve-shortening property:
\begin{align}
	\frac{\mathcal{L}[\Gamma^{n+1}]-\mathcal{L}[\Gamma^n]}{\Delta t^n}
	=
	-\sum_{i=1}^N(\kappa_i^{n+1/2})^2r_i^{n+1/2}
	<
	0.
\end{align}
\end{ex}
    
\begin{ex}[fully discrete polygonal area-preserving mean curvature flow]
\label{ex:time-discretized_area-preserving_mean_curvature_flow_pmbp}
The normal velocity of the fully discrete polygonal area-preserving mean curvature flow is given by
\begin{equation}
	v_i^{n+1/2}
	=
	-\kappa_i^{n+1/2}+\langle\kappa^{n+1/2}\rangle_{\Gamma^{n+1/2}},
	\quad
	i=1,2,\ldots,N,\ n=0,1,\ldots.
\end{equation}
Then, we have the curve-shortening property:
\begin{align}
	&\frac{\mathcal{L}[\Gamma^{n+1}]-\mathcal{L}[\Gamma^n]}{\Delta t^n}
	=
	\sum_{i=1}^N\kappa_i^{n+1/2}\left(-\kappa_i^{n+1/2}+\langle\kappa^{n+1/2}\rangle_{\Gamma^{n+1/2}}\right)r_i^{n+1/2}\\
	&\hspace{5pt}=
	-\sum_{i=1}^N(\kappa_i^{n+1/2})^2r_i^{n+1/2}+\frac{1}{\mathcal{L}[\Gamma^{n+1/2}]}\left(\sum_{i=1}^N\kappa_i^{n+1/2}r_i^{n+1/2}\right)^2\\
	&\hspace{5pt}=
	\frac{1}{\mathcal{L}[\Gamma^{n+1/2}]}\left[-\sum_{i=1}^Nr_i^{n+1/2}\sum_{i=1}^N(\kappa_i^{n+1/2})^2r_i^{n+1/2}+\left(\sum_{i=1}^N\kappa_i^{n+1/2}r_i^{n+1/2}\right)^2\right]
	\le
	0.
\end{align}
\end{ex}
    
\begin{ex}[fully discrete polygonal Hele-Shaw flow]
\label{ex:time-discretized_Hele-Shaw_flow_pmbp}
The normal velocity of the fully discrete polygonal Hele-Shaw flow is given by
\begin{equation}
	v_i^{n+1/2}=-\nabla p^{n+1/2}\cdot\bm{n}_i^{n+1/2},
	\quad
	i=1,2,\ldots,N,\ n=0,1,\ldots,
\end{equation}
where $p^{n+1/2}$ is the solution to the Laplace equation
\begin{equation}
	\begin{dcases*}
		\triangle p^{n+1/2}=0&in $\Omega^{n+1/2}$,\\
		p^{n+1/2}=\sigma\kappa_i^{n+1/2}&on $\Gamma_i^{n+1/2}$, $i=1,2,\ldots,N$,\\
		\nabla p^{n+1/2}\cdot\bm{n}_i^{n+1/2}\equiv\text{const.}&on $\Gamma_i^{n+1/2}$, $i=1,2,\ldots,N$.
	\end{dcases*}
\label{eq:Hele-Shaw_mid}
\end{equation}
Then, we have the curve-shortening property:
\begin{align}
	\frac{\mathcal{L}[\Gamma^{n+1}]-\mathcal{L}[\Gamma^n]}{\Delta t^n}
	&=
	-\frac{1}{\sigma}\sum_{i=1}^Np^{n+1/2}|_{\Gamma_i^{n+1/2}}\nabla p^{n+1/2}|_{\Gamma_i^{n+1/2}}\cdot\bm{n}_i^{n+1/2}r_i^{n+1/2}\\
	&=
	-\frac{1}{\sigma}\int_{\Gamma^{n+1/2}}p^{n+1/2}\nabla p^{n+1/2}\cdot\bm{n}^{n+1/2}\,\mathrm{d}s\\
	&=
	-\frac{1}{\sigma}\int_{\Omega^{n+1/2}}\nabla\cdot(p^{n+1/2}\nabla p^{n+1/2})\,\mathrm{d}\bm{x}\\
	&=
	-\frac{1}{\sigma}\int_{\Omega^{n+1/2}}|\nabla p^{n+1/2}|^2\,\mathrm{d}\bm{x}
	\le
	0,
\end{align}
where $\bm{n}^{n+1/2}:=\sum_{i=1}^N\bm{n}_i^{n+1/2}\chi_{\Gamma_i^{n+1/2}}$.
\end{ex}
    
\begin{rem}
The proposed fully discrete evolution law does not inherit the area-preservation/dissipation property exactly.
The behavior of the enclosed area will be checked numerically in Section~\ref{sec:numerical_experiments}.
A fully discrete polygonal evolution law that inherits the property will be given in Appendix \ref{sec:area}.
\end{rem}
    
\section{Numerical experiments}
\label{sec:numerical_experiments}

\subsection{Results}

We give several numerical experiments to illustrate the qualitative behavior of the proposed fully discrete evolution law.
All were carried out by using Julia 1.1.0 on a machine with 3.1 GHz Intel Core i5, 16GB memory, OS X 10.14.5. 
Nonlinear equations are solved by \textsf{nlsolve}\footnote{The function \textsf{nlsolve} is a typical nonlinear solver in Julia. \url{https://pkg.julialang.org/docs/ NLsolve/}} with residual tolerance $10^{-8}$.
    
\subsubsection*{Computation of tangential velocity}

In Section \ref{subsec:tangential_velocity}, we have developed the asymptotic uniform distribution method.
We employ the method with the simple zero average condition \eqref{eq:zero_average_condition}.
The initial vertices are arranged uniformly.
    
\subsection*{Time step size}
    
We use adaptive time step sizes.
Since all the target problems are gradient flows of the length functional on some space, the time step size should be smaller than the time derivative of the length functional so that the discrete flow captures key dynamics of the original moving boundary problem.
Based on this idea, we control the time step size by the formula
\begin{align}
	\Delta t^{n+1}
	&:=
	\min\left\{\tau,\left(\frac{\mathcal{L}[\Gamma^{n+1}]-\mathcal{L}[\Gamma^{n}]}{\Delta t^{n}}\right)^{-2}\right\},
	\quad
	n=0,1,\ldots,\\
	\Delta t^0
	&:=
	\min\left\{\tau,\left(\sum_{i=1}^N(\kappa_i^0)^2r_i^0\right)^{-2}\right\},
\end{align}
where $\kappa_i^0$ is the discrete curvature of the initial curve $\Gamma^0$ defined in \eqref{eq:discrete_curvature} (see~\cite{K17}: for more details).
Note that a general-purpose explicit time integrator, such as the Runge--Kutta method, requires sufficiently small time step sizes as mentioned in Section \ref{sec:introduction}.
Thus the above formula makes sense only for a numerical integrator that allows relatively large time step sizes.
    
\subsection*{Initial curve and parameters}

The initial curve is set to
\begin{equation}
	x_1(t)=0.5a_1(t),
	\quad
	x_2(t)=0.54a_3(t),
	\quad
	t\in[0,1],
\label{eq:initial_curve}
\end{equation}
where
\begin{align}
	a_1(t)&=1.8\cos(2\pi t),
    \quad
	a_2(t)=0.2+\sin(\pi t)\sin(6\pi t)\sin(2a_1(t)),\\
	a_3(t)&=0.5\sin(2\pi t)+\sin a_1(t)+a_2(t)\sin(2\pi t),
	\quad
	t\in[0,1].
\end{align}
We set $N$ vertices $\tilde{\bm{X}}_i^0=(\tilde{X}_{i,1}^0,\tilde{X}_{i,2}^0)^{\mathrm{T}}$ to 
\begin{equation}
	\tilde{X}_{i,1}^0=x_1(i/N),
	\quad
	\tilde{X}_{i,2}^0=x_2(i/N),
	\quad
	i=1,2,\ldots,N.
	\label{eq:original_initial_vertices}
\end{equation}
Since the distribution of the vertices is not uniform, we modify them by performing the asymptotic uniform distribution method with all the normal velocities being equal to 0 until the vertices are arranged uniformly.
The resulted vertices, denoted by $\bm{X}_i^0$, are used as the initial curve.
Figure \ref{fig:initial_vertices} compares $\bm{X}^0$ with $\tilde{\bm{X}}^0$.
It is observed that the initial vertices after performing the asymptotic uniform distribution method slant a little from the original curve; however, this does not be a matter in a practical situation.

\begin{figure}[tb]
\centering
\begin{tikzpicture}
\tikzstyle{every node}=[font=\scriptsize]
\begin{axis}[width=0.53\hsize,height=0.53\hsize,%restrict x to domain = 0:69,
xmax=1,xmin=-1,
ymax=1, ymin = -1,
title style={font=\normalsize},
title= (a) original arrangement $\tilde{\bm{X}}^0$,
	]
\addplot[domain=0:{1}, variable=\t, samples=100] (
{0.9*cos(2*pi*\t r)},
{0.54*(0.5*sin(2*pi*\t r)+sin(1.8*cos(2*pi*\t r) r)+(0.2+sin(pi*\t r)*sin(6*pi*\t r)*sin(2*1.8*cos(2*pi*\t r) r))*sin(2*pi*\t r))}
);
\addplot[only marks,mark=o,mark size=1pt
] table {
0.892903 0.574290
0.871725 0.622076
0.836799 0.669419
0.788676 0.721245
0.728115 0.779830
0.656072 0.832649
0.573682 0.847505
0.482244 0.788687
0.383201 0.651479
0.278115 0.485492
0.168643 0.371707
0.056511 0.356935
-0.056511 0.402843
-0.168643 0.408130
-0.278115 0.293397
-0.383201 0.070030
-0.482244 -0.171491
-0.573682 -0.336939
-0.656072 -0.394244
-0.728115 -0.379486
-0.788676 -0.354588
-0.836799 -0.361927
-0.871725 -0.407095
-0.892903 -0.469824
-0.900000 -0.525878
-0.892903 -0.564576
-0.871725 -0.595105
-0.836799 -0.640229
-0.788676 -0.718794
-0.728115 -0.823851
-0.656072 -0.911762
-0.573682 -0.919447
-0.482244 -0.809803
-0.383201 -0.614019
-0.278115 -0.425602
-0.168643 -0.334479
-0.056511 -0.351665
0.056511 -0.397573
0.168643 -0.370903
0.278115 -0.233507
0.383201 -0.032571
0.482244 0.150375
0.573682 0.264997
0.656072 0.315131
0.728115 0.335465
0.788676 0.357039
0.836799 0.391117
0.871725 0.434067
0.892903 0.479538
0.900000 0.525878
};
\end{axis}
\end{tikzpicture}
\begin{tikzpicture}
\tikzstyle{every node}=[font=\scriptsize]
\begin{axis}[width=0.53\hsize,height=0.53\hsize,%restrict x to domain = 0:69,
xmax=1,xmin=-1,
ymax=1, ymin = -1,
title style={font=\normalsize},
title= (b) initial arrangement $\bm{X}^0$,
	]
\addplot[domain=0:{1}, variable=\t, samples=100] (
{0.9*cos(2*pi*\t r)},
{0.54*(0.5*sin(2*pi*\t r)+sin(1.8*cos(2*pi*\t r) r)+(0.2+sin(pi*\t r)*sin(6*pi*\t r)*sin(2*1.8*cos(2*pi*\t r) r))*sin(2*pi*\t r))}
);
\addplot[only marks,mark=o,mark size=1pt
] table {
0.872881 0.624209
0.801662 0.710895
0.714591 0.781644
0.614303 0.831933
0.502292 0.825583
0.421516 0.747724
0.376509 0.644957
0.316764 0.549998
0.248150 0.461235
0.168769 0.381955
0.062465 0.346091
-0.044119 0.381113
-0.152863 0.408710
-0.250314 0.353122
-0.301285 0.253178
-0.348278 0.151304
-0.388402 0.046534
-0.435968 -0.055075
-0.478087 -0.159059
-0.505466 -0.267858
-0.566286 -0.362132
-0.678381 -0.357508
-0.789921 -0.345440
-0.873531 -0.420246
-0.899750 -0.529330
-0.855033 -0.632224
-0.781244 -0.716733
-0.722580 -0.812365
-0.641513 -0.889920
-0.530968 -0.909067
-0.453706 -0.827720
-0.439328 -0.716455
-0.384540 -0.618552
-0.332303 -0.519264
-0.275591 -0.422463
-0.195748 -0.343648
-0.083563 -0.344808
0.018231 -0.391976
0.130407 -0.390222
0.215643 -0.317273
0.278434 -0.224299
0.331968 -0.125705
0.384016 -0.026318
0.433474 0.074383
0.484531 0.174282
0.561866 0.255559
0.664218 0.301503
0.768661 0.342472
0.857210 0.411362
0.899584 0.515242
};
\end{axis}
\end{tikzpicture}
\caption{Arrangements of the initial vertices.
        The solid lines represent the original initial curve defined by \eqref{eq:initial_curve}.
        The circles represent the positions of the original vertices defined by \eqref{eq:original_initial_vertices} in (a) and the ones of the vertices after performing the asymptotic uniform distribution method in (b).}
\label{fig:initial_vertices}
\end{figure}
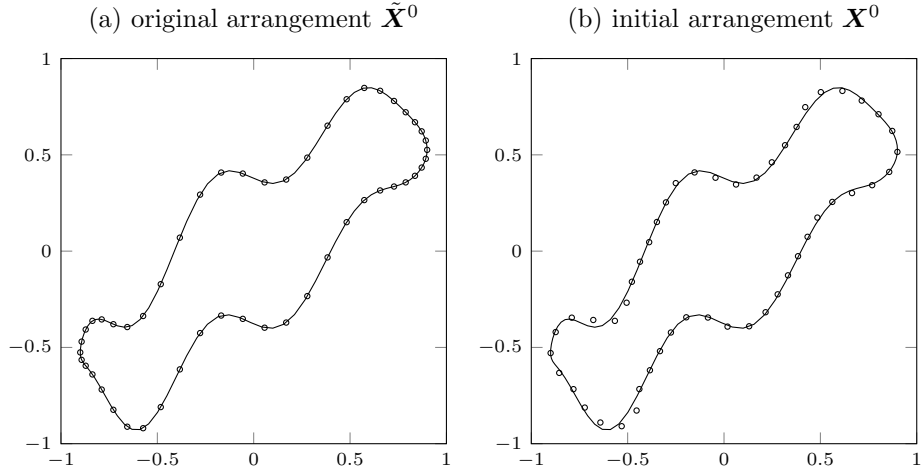
	
The number $N$ of vertices, the parameter $\tau$, and the relaxation parameter $\omega$ are respectively set to $50$, $0.01$, and $10N/\Delta t^n$ for all numerical experiments.

%\subsection*{Solver for the modified polygonal evolution law}
%    
%Since the fully discrete polygonal evolution law \eqref{eq:discretized_polygonal_evolution_law} is implicit, we adopt the Newton method to obtain the solution $\Gamma^n$ at each time step with the adaptive time step size $\Delta t^{n-1}$.
    
\begin{ex}[mean curvature flow]
Our first example is the mean curvature flow (see Examples \ref{ex:mean_curvature_flow}, \ref{ex:mean_curvature_flow_pmbp}, and \ref{ex:time-discretized_mean_curvature_flow_pmbp}).
It is well known that the solution curve to the mean curvature flow becomes convex in finite time and shrinks to a point with converging to a circle \cite{GH86,G87}.
Figure \ref{fig:MCF} shows the evolution of the flow and the behavior of the length and enclosed area.
It is observed that the curve converges to a circle, and the size becomes small as time passes.
The length monotonically decreases.
Further, the area-dissipation property is also captured well: the area decreases at an almost constant rate.
Figure~\ref{fig:MCF_stepsize} shows the time step sizes selected during the numerical integration.
Our preliminary experiments suggested that the time step sizes should be smaller than $0.1N^{-2}$ ($=4.0\times 10^{-5}$ in our problem setting) for the Runge--Kutta method.
In comparison with the Runge--Kutta method, larger step sizes are selected for the proposed evolution law.

\input{MCF.tex}
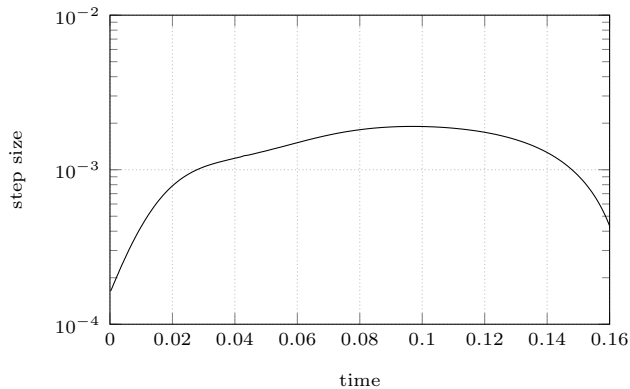
\begin{figure}[tb]
\centering
\begin{tikzpicture}[baseline]
\tikzstyle{every node}=[font=\scriptsize]
\begin{semilogyaxis}[width=0.65\hsize,height=0.45\hsize,%restrict x to domain = 0:69,
xmax=0.16,xmin=0,
ymax=0.01, ymin = 0.0001,
scaled ticks=false,
tick label style={/pgf/number format/fixed},
grid = major,
grid style={densely dotted},
xlabel = time,
ylabel = step size,
	]
\addplot[thin
] table {
0.000164	0.000164
0.00033	0.000166
0.000499	0.000169
0.000671	0.000172
0.000846	0.000175
0.001025	0.000179
0.001208	0.000183
0.001394	0.000186
0.001584	0.00019
0.001778	0.000194
0.001977	0.000199
0.00218	0.000203
0.002388	0.000208
0.0026	0.000212
0.002817	0.000217
0.00304	0.000223
0.003268	0.000228
0.003502	0.000234
0.003742	0.00024
0.003988	0.000246
0.00424	0.000252
0.004499	0.000259
0.004765	0.000266
0.005038	0.000273
0.005319	0.000281
0.005608	0.000289
0.005906	0.000298
0.006213	0.000307
0.006529	0.000316
0.006855	0.000326
0.007191	0.000336
0.007538	0.000347
0.007896	0.000358
0.008266	0.00037
0.008649	0.000383
0.009045	0.000396
0.009455	0.00041
0.009879	0.000424
0.010318	0.000439
0.010773	0.000455
0.011245	0.000472
0.011734	0.000489
0.012242	0.000508
0.012769	0.000527
0.013316	0.000547
0.013884	0.000568
0.014474	0.00059
0.015087	0.000613
0.015723	0.000636
0.016383	0.00066
0.017068	0.000685
0.017778	0.00071
0.018514	0.000736
0.019276	0.000762
0.020064	0.000788
0.020878	0.000814
0.021718	0.00084
0.022585	0.000867
0.023477	0.000892
0.024394	0.000917
0.025335	0.000941
0.0263	0.000965
0.027287	0.000987
0.028296	0.001009
0.029325	0.001029
0.030374	0.001049
0.031441	0.001067
0.032525	0.001084
0.033626	0.001101
0.034743	0.001117
0.035875	0.001132
0.037022	0.001147
0.038185	0.001163
0.039363	0.001178
0.040557	0.001194
0.041766	0.001209
0.042998	0.001232
0.04424	0.001242
0.045499	0.001259
0.046776	0.001277
0.048073	0.001297
0.049388	0.001315
0.050724	0.001336
0.052083	0.001359
0.053464	0.001381
0.054868	0.001404
0.056297	0.001429
0.057752	0.001455
0.059234	0.001482
0.060743	0.001509
0.062281	0.001538
0.063848	0.001567
0.065444	0.001596
0.067069	0.001625
0.068723	0.001654
0.070405	0.001682
0.072115	0.00171
0.073851	0.001736
0.075612	0.001761
0.077396	0.001784
0.079202	0.001806
0.081028	0.001826
0.082872	0.001844
0.084732	0.00186
0.086605	0.001873
0.088489	0.001884
0.090382	0.001893
0.092282	0.0019
0.094186	0.001904
0.096092	0.001906
0.097998	0.001906
0.099902	0.001904
0.101801	0.001899
0.103694	0.001893
0.105578	0.001884
0.107452	0.001874
0.109313	0.001861
0.111159	0.001846
0.112989	0.00183
0.1148	0.001811
0.116591	0.001791
0.11836	0.001769
0.120105	0.001745
0.121824	0.001719
0.123516	0.001692
0.125178	0.001662
0.126809	0.001631
0.128408	0.001599
0.129973	0.001565
0.131503	0.00153
0.132996	0.001493
0.134452	0.001456
0.135869	0.001417
0.137246	0.001377
0.138583	0.001337
0.139879	0.001296
0.141134	0.001255
0.142347	0.001213
0.143518	0.001171
0.144648	0.00113
0.145736	0.001088
0.146783	0.001047
0.147789	0.001006
0.148755	0.000966
0.149682	0.000927
0.15057	0.000888
0.151421	0.000851
0.152235	0.000814
0.153013	0.000778
0.153756	0.000743
0.154466	0.00071
0.155143	0.000677
0.155789	0.000646
0.156404	0.000615
0.15699	0.000586
0.157548	0.000558
0.158079	0.000531
0.158585	0.000506
0.159066	0.000481
0.159523	0.000457
0.159958	0.000435
0.160371	0.000413
0.160764	0.000393
};
\end{semilogyaxis}
\end{tikzpicture}
\caption{Time step sizes used for the mean curvature flow.}
\label{fig:MCF_stepsize}
\end{figure}
\end{ex}
    
\begin{ex}
The next example is the area-preserving mean curvature flow (see Examples \ref{ex:area-preserving_mean_curvature_flow}, \ref{ex:area-preserving_mean_curvature_flow_pmbp}, and \ref{ex:time-discretized_area-preserving_mean_curvature_flow_pmbp}).
Theoretically, the solution curve converges to a circle, but does not converge to a point due to the area-preservation~\cite{G86}.
Figure \ref{fig:APMCF} shows the numerical results, from which the expected behavior is observed.
We note that the enclosed are is almost constant.
Figure~\ref{fig:APMCF_stepsize} shows the time step sizes selected during the numerical integration.
In comparison with the previous example for the mean curvature flow (Figure~\ref{fig:MCF_stepsize}), much larger step step sizes are selected.

\input{APMCF.tex}
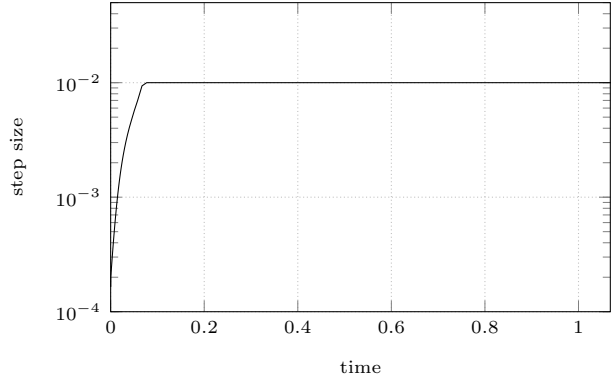
\begin{figure}[tb]
\centering
\begin{tikzpicture}[baseline]
\tikzstyle{every node}=[font=\scriptsize]
\begin{semilogyaxis}[width=0.65\hsize,height=0.45\hsize,%restrict x to domain = 0:69,
xmax=1.068,xmin=0,
ymax=0.05, ymin = 0.0001,
scaled ticks=false,
tick label style={/pgf/number format/fixed},
grid = major,
grid style={densely dotted},
xlabel = time,
ylabel = step size,
	]
\addplot[thin
] table {
0.000164	0.000164
0.000367	0.000203
0.000575	0.000208
0.000789	0.000214
0.00101	0.000221
0.001237	0.000227
0.001471	0.000234
0.001713	0.000242
0.001963	0.00025
0.002221	0.000258
0.002488	0.000267
0.002765	0.000277
0.003052	0.000287
0.00335	0.000298
0.00366	0.00031
0.003983	0.000323
0.00432	0.000337
0.004672	0.000352
0.00504	0.000368
0.005426	0.000386
0.005831	0.000405
0.006257	0.000426
0.006705	0.000448
0.007179	0.000474
0.00768	0.000501
0.008212	0.000532
0.008778	0.000566
0.009383	0.000605
0.01003	0.000647
0.010725	0.000695
0.011475	0.00075
0.012286	0.000811
0.013168	0.000882
0.01413	0.000962
0.015185	0.001055
0.016347	0.001162
0.017633	0.001286
0.019065	0.001432
0.020668	0.001603
0.022472	0.001804
0.024513	0.002041
0.026834	0.002321
0.029487	0.002653
0.032534	0.003047
0.036053	0.003519
0.040138	0.004085
0.044912	0.004774
0.050629	0.005717
0.057608	0.006979
0.066984	0.009376
0.076984	0.01
0.086984	0.01
0.096984	0.01
0.106984	0.01
0.116984	0.01
0.126984	0.01
0.136984	0.01
0.146984	0.01
0.156984	0.01
0.166984	0.01
0.176984	0.01
0.186984	0.01
0.196984	0.01
0.206984	0.01
0.216984	0.01
0.226984	0.01
0.236984	0.01
0.246984	0.01
0.256984	0.01
0.266984	0.01
0.276984	0.01
0.286984	0.01
0.296984	0.01
0.306984	0.01
0.316984	0.01
0.326984	0.01
0.336984	0.01
0.346984	0.01
0.356984	0.01
0.366984	0.01
0.376984	0.01
0.386984	0.01
0.396984	0.01
0.406984	0.01
0.416984	0.01
0.426984	0.01
0.436984	0.01
0.446984	0.01
0.456984	0.01
0.466984	0.01
0.476984	0.01
0.486984	0.01
0.496984	0.01
0.506984	0.01
0.516984	0.01
0.526984	0.01
0.536984	0.01
0.546984	0.01
0.556984	0.01
0.566984	0.01
0.576984	0.01
0.586984	0.01
0.596984	0.01
0.606984	0.01
0.616984	0.01
0.626984	0.01
0.636984	0.01
0.646984	0.01
0.656984	0.01
0.666984	0.01
0.676984	0.01
0.686984	0.01
0.696984	0.01
0.706984	0.01
0.716984	0.01
0.726984	0.01
0.736984	0.01
0.746984	0.01
0.756984	0.01
0.766984	0.01
0.776984	0.01
0.786984	0.01
0.796984	0.01
0.806984	0.01
0.816984	0.01
0.826984	0.01
0.836984	0.01
0.846984	0.01
0.856984	0.01
0.866984	0.01
0.876984	0.01
0.886984	0.01
0.896984	0.01
0.906984	0.01
0.916984	0.01
0.926984	0.01
0.936984	0.01
0.946984	0.01
0.956984	0.01
0.966984	0.01
0.976984	0.01
0.986984	0.01
0.996984	0.01
1.006984	0.01
1.016984	0.01
1.026984	0.01
1.036984	0.01
1.046984	0.01
1.056984	0.01
1.066984	0.01
1.076984	0.01
};
\end{semilogyaxis}
\end{tikzpicture}
\caption{Time step sizes used for the area-preserving mean curvature flow.}
\label{fig:APMCF_stepsize}
\end{figure}
\end{ex}
    
\begin{ex}
Our final example is the Hele-Shaw flow (see Examples \ref{ex:Hele-Shaw_flow}, \ref{ex:Hele-Shaw_flow_pmbp}, and \ref{ex:time-discretized_Hele-Shaw_flow_pmbp}).
To solve the potential problem (the Laplace equation) \eqref{eq:Hele-Shaw_mid}, we employ the method of fundamental solutions (MFS)~\cite{SY19}.
Briefly speaking, we approximate a function $p^{n+1/2}$ by a linear combination of the fundamental solution $E(\bm{x})=(2\pi)^{-1}\log|\bm{x}|$ of the Laplace operator:
\begin{equation}
	p^{n+1/2}(\bm{x})
	\approx
	P^{n+1/2}(\bm{x})
	:=
	Q_0^{n+1/2}+\sum_{j=1}^NQ_j^{n+1/2}E(\bm{x}-\bm{y}_j^{n+1/2}),
\end{equation}
where $\{\bm{y}_j^{n+1/2}\}_{j=1}^N$ are the singular points suitably chosen from $\mathbb{R}^2\setminus\overline\Omega^{n+1/2}$.
The coefficients $\{Q_j^{n+1/2}\}_{j=0}^N$ are determined by the collocation method.
More precisely, we use the points $\bm{X}_{i,\mathrm{mid}}^{n+1/2}=(\bm{X}_i^{n+1/2}+\bm{X}_{i-1}^{n+1/2})/2$ as the collocation points, and solve the following collocation equations:
\begin{equation}
	P^{n+1/2}(\bm{X}_{i,\mathrm{mid}}^{n+1/2})=\sigma\kappa_i^{n+1/2},
	\quad
	i=1,2,\ldots,N.
\end{equation}
%The strong feature of the MFS is that we can compute its derivative, so the normal velocity, analytically.
Our numerical results are depicted in Figure \ref{fig:Hele-Shaw}.
The solution curve converges to a circle.
The length monotonically decreases, and the area-preserving property is also captured well.
Figure~\ref{fig:HS_stepsize} shows the time step sizes selected during the numerical integration.
The selected step step sizes reach $\tau = 0.01$ soon after the numerical integration begins.

\input{Hele-Shaw.tex}
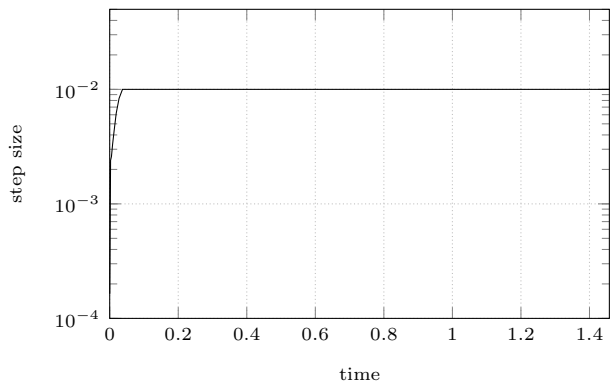
\begin{figure}[tb]
\centering
\begin{tikzpicture}[baseline]
\tikzstyle{every node}=[font=\scriptsize]
\begin{semilogyaxis}[width=0.65\hsize,height=0.45\hsize,%restrict x to domain = 0:69,
xmax=1.458,xmin=0,
ymax=0.05, ymin = 0.0001,
scaled ticks=false,
tick label style={/pgf/number format/fixed},
grid = major,
grid style={densely dotted},
xlabel = time,
ylabel = step size,
	]
\addplot[thin
] table {
0.000164	0.000164
0.002564	0.0024
0.005163	0.002599
0.00848	0.003317
0.0128	0.00432
0.018975	0.006175
0.027309	0.008334
0.037309	0.01
0.047309	0.01
0.057309	0.01
0.067309	0.01
0.077309	0.01
0.087309	0.01
0.097309	0.01
0.107309	0.01
0.117309	0.01
0.127309	0.01
0.137309	0.01
0.147309	0.01
0.157309	0.01
0.167309	0.01
0.177309	0.01
0.187309	0.01
0.197309	0.01
0.207309	0.01
0.217309	0.01
0.227309	0.01
0.237309	0.01
0.247309	0.01
0.257309	0.01
0.267309	0.01
0.277309	0.01
0.287309	0.01
0.297309	0.01
0.307309	0.01
0.317309	0.01
0.327309	0.01
0.337309	0.01
0.347309	0.01
0.357309	0.01
0.367309	0.01
0.377309	0.01
0.387309	0.01
0.397309	0.01
0.407309	0.01
0.417309	0.01
0.427309	0.01
0.437309	0.01
0.447309	0.01
0.457309	0.01
0.467309	0.01
0.477309	0.01
0.487309	0.01
0.497309	0.01
0.507309	0.01
0.517309	0.01
0.527309	0.01
0.537309	0.01
0.547309	0.01
0.557309	0.01
0.567309	0.01
0.577309	0.01
0.587309	0.01
0.597309	0.01
0.607309	0.01
0.617309	0.01
0.627309	0.01
0.637309	0.01
0.647309	0.01
0.657309	0.01
0.667309	0.01
0.677309	0.01
0.687309	0.01
0.697309	0.01
0.707309	0.01
0.717309	0.01
0.727309	0.01
0.737309	0.01
0.747309	0.01
0.757309	0.01
0.767309	0.01
0.777309	0.01
0.787309	0.01
0.797309	0.01
0.807309	0.01
0.817309	0.01
0.827309	0.01
0.837309	0.01
0.847309	0.01
0.857309	0.01
0.867309	0.01
0.877309	0.01
0.887309	0.01
0.897309	0.01
0.907309	0.01
0.917309	0.01
0.927309	0.01
0.937309	0.01
0.947309	0.01
0.957309	0.01
0.967309	0.01
0.977309	0.01
0.987309	0.01
0.997309	0.01
1.007309	0.01
1.017309	0.01
1.027309	0.01
1.037309	0.01
1.047309	0.01
1.057309	0.01
1.067309	0.01
1.077309	0.01
1.087309	0.01
1.097309	0.01
1.107309	0.01
1.117309	0.01
1.127309	0.01
1.137309	0.01
1.147309	0.01
1.157309	0.01
1.167309	0.01
1.177309	0.01
1.187309	0.01
1.197309	0.01
1.207309	0.01
1.217309	0.01
1.227309	0.01
1.237309	0.01
1.247309	0.01
1.257309	0.01
1.267309	0.01
1.277309	0.01
1.287309	0.01
1.297309	0.01
1.307309	0.01
1.317309	0.01
1.327309	0.01
1.337309	0.01
1.347309	0.01
1.357309	0.01
1.367309	0.01
1.377309	0.01
1.387309	0.01
1.397309	0.01
1.407309	0.01
1.417309	0.01
1.427309	0.01
1.437309	0.01
1.447309	0.01
1.457309	0.01
1.467309	0.01
};
\end{semilogyaxis}
\end{tikzpicture}
\caption{Time step sizes used for the Hele-Shaw flow.}
\label{fig:HS_stepsize}
\end{figure}
\end{ex}

\subsection{Discussion}

The above numerical experiments confirm that the proposed method performs well.
For the remainder of this section, we focus on the mean curvature flow and explore properties of the proposed method.
% In the last part of this section, we will discuss in more depth using the mean curvature flow.

\subsubsection*{Comparison with the Runge--Kutta method}

We check the correctness of the proposed method.
To this end, we compare the numerical solution of the proposed method with the reference solution obtained by the usual fourth-order Runge--Kutta method with much smaller time step size applied to the semi-discrete polygonal mean curvature flow (Example \ref{ex:mean_curvature_flow_pmbp}).
The time step size for the proposed method is adaptively controlled as done in the previous subsection, while the uniform step size $\Delta t=0.1N^{-2}$ is employed for the Runge--Kutta method, where $N=50$.
Both results are displayed in Figure~\ref{fig:MCF_comparison}, from which we observe that the asymptotic behavior of our method is consistent with that of the Runge--Kutta method.
\input{MCF_comparison}

% First, we verify the correctness of the results obtained by our method by comparing them with those obtained by well-known numerical methods.
% To this end, we consider the semi-discrete polygonal mean curvature flow (Example \ref{ex:mean_curvature_flow_pmbp}) and apply the usual fourth-order Runge--Kutta method for time discretization.
% We show the obtained results in Figure~\ref{fig:MCF_comparison}.
% \input{MCF_comparison}
% As mentioned in the Introduction (Section \ref{sec:introduction}), the time increments must be small enough to be stable when using a general-purpose method such as the Runge--Kutta method.
% Here, a uniform time step size $\Delta t=0.1N^{-2}$ is employed in the computation.
% The asymptotic behavior of our method's results is consistent with that of the Runge--Kutta method, which shows the correctness of our method, as confirmed in Figure \ref{fig:MCF_comparison}.

\subsubsection*{Robustness}

We check the robustness of the proposed method with respect to the time step size.
We here consider the uniform time step size and set it to $\Delta t = 0.01$, which is much bigger than the step size selected by the adaptive method.
The results for both the proposed method and the Runge--Kutta method are shown in Figure \ref{fig:MCF_large_time_step}, from which we observe a significant difference.
From the right figure (b), we see that the Runge--Kutta method is unstable even in a single time stepping (the dotted curve represents the numerical solution after one step).
In contrast, the left figure (a), which is almost consistent with the previous computations, indicates that the proposed method performs better for a longer time interval.
We may conclude that the proposed method is more robust than the Runge--Kutta method.
However, the proposed method may produce a wrong solution if the computation proceeds with this step size to later times. 
Indeed, as shown in Figure \ref{fig:MCF_later_times}, the numerical solution becomes unstable if we continue the numerical computation from the final state in Figure~\ref{fig:MCF_large_time_step}(a).
Therefore, we recommend using the adaptive method.
\input{MCF_large_time_step}
\begin{figure}[tb]
\centering
\begin{tikzpicture}[baseline]
\tikzstyle{every node}=[font=\scriptsize]
\begin{axis}[width=0.49\hsize,height=0.49\hsize,%restrict x to domain = 0:69,
xmax=0.5,xmin=-0.5,
ymax=0.5, ymin = -0.5,
	]
\addplot[thin,smooth
] table {
 0.457838 0.324568
 0.432084 0.372699
 0.391679 0.409455
 0.342267 0.433945
 0.288815 0.443754
 0.230501 0.443269
 0.171195 0.431085
 0.131334 0.409313
 0.075431 0.389671
 0.028484 0.359997
 -0.019730 0.334914
 -0.065222 0.301955
 -0.109372 0.270617
 -0.151675 0.235348
 -0.192635 0.198803
 -0.232395 0.160761
 -0.267813 0.118906
 -0.310438 0.076551
 -0.339599 0.031615
 -0.376763 -0.011287
 -0.401317 -0.060807
 -0.431574 -0.107941
 -0.446750 -0.158660
 -0.462509 -0.204971
 -0.463875 -0.267777
 -0.454078 -0.321489
 -0.426041 -0.370143
 -0.385756 -0.406063
 -0.335535 -0.426809
 -0.281613 -0.433904
 -0.228364 -0.428690
 -0.177813 -0.416747
 -0.128169 -0.397809
 -0.077937 -0.374579
 -0.029194 -0.346894
 0.018067 -0.318415
 0.064417 -0.285483
 0.110634 -0.253483
 0.153274 -0.217741
 0.195916 -0.183766
 0.235023 -0.146339
 0.278729 -0.104892
 0.315105 -0.064062
 0.351786 -0.020596
 0.382422 0.019524
 0.412229 0.068858
 0.437453 0.114457
 0.457452 0.168655
 0.468918 0.221033
 0.469517 0.275551
 0.457838 0.324568
};
\addplot[thin,dotted,smooth
] table {
0.413716 0.292058
0.389205 0.335600
0.353068 0.370164
0.308485 0.393931
0.260347 0.406479
0.206694 0.406792
0.151056 0.400019
0.112532 0.386591
0.061823 0.366513
0.017225 0.341861
-0.026437 0.318441
-0.068584 0.288707
-0.108796 0.259835
-0.147795 0.227603
-0.185017 0.193841
-0.221988 0.159540
-0.255016 0.121803
-0.293432 0.081791
-0.320360 0.040945
-0.352980 0.000267
-0.374804 -0.045430
-0.399808 -0.090336
-0.412106 -0.137064
-0.423532 -0.179870
-0.420336 -0.237993
-0.408356 -0.286511
-0.381069 -0.330231
-0.344498 -0.363371
-0.299447 -0.384404
-0.250654 -0.394212
-0.201781 -0.393004
-0.155093 -0.385248
-0.108860 -0.370540
-0.061831 -0.351532
-0.016452 -0.327327
0.027403 -0.302183
0.070342 -0.272390
0.112381 -0.242596
0.151866 -0.210225
0.190257 -0.178565
0.226236 -0.144680
0.265877 -0.105415
0.299335 -0.068183
0.331581 -0.027041
0.358586 0.010026
0.384546 0.056265
0.405123 0.099070
0.421106 0.149759
0.428438 0.198210
0.426112 0.248060
0.413716 0.292058
};
\addplot[thin,dotted,smooth
] table {
0.369823 0.259564
0.346960 0.298689
0.314571 0.330427
0.274899 0.353450
0.231759 0.366628
0.182981 0.371035
0.131783 0.365974
0.097000 0.356906
0.048783 0.342622
0.007615 0.321590
-0.033730 0.302481
-0.071068 0.274084
-0.108687 0.250809
-0.143556 0.219066
-0.178988 0.189891
-0.211831 0.158464
-0.241526 0.124065
-0.276710 0.086863
-0.301151 0.050424
-0.331981 0.009642
-0.354018 -0.030177
-0.374772 -0.073802
-0.379693 -0.109860
-0.376942 -0.160319
-0.388246 -0.202196
-0.364717 -0.257148
-0.341196 -0.295290
-0.304397 -0.323982
-0.264645 -0.344991
-0.220203 -0.354838
-0.176022 -0.356740
-0.133534 -0.352041
-0.091018 -0.341290
-0.047510 -0.326121
-0.005455 -0.305657
0.034989 -0.283973
0.074619 -0.257604
0.113277 -0.231118
0.149209 -0.201770
0.183601 -0.172593
0.216057 -0.141768
0.251487 -0.104903
0.281355 -0.070725
0.309426 -0.032262
0.332401 0.001932
0.354655 0.044877
0.370892 0.084510
0.382939 0.131480
0.386863 0.175645
0.382328 0.220666
0.369823 0.259564
};
\addplot[thin,smooth
] table {
 0.328212 0.227569
 0.307940 0.262742
 0.280103 0.292818
 0.246624 0.317296
 0.211326 0.334177
 0.173204 0.346176
 0.116417 0.353657
 0.083394 0.343848
 0.032050 0.330519
 0.000104 0.300067
 -0.038541 0.284362
 -0.072595 0.258418
 -0.106800 0.238447
 -0.138894 0.210533
 -0.170900 0.184961
 -0.201620 0.156210
 -0.229078 0.127249
 -0.259887 0.091509
 -0.282727 0.063165
 -0.307827 0.019999
 -0.330908 -0.013545
 -0.330948 -0.064189
 -0.364696 -0.093947
 -0.374011 -0.141172
 -0.371147 -0.174846
 -0.330350 -0.237501
 -0.300037 -0.261461
 -0.274708 -0.300853
 -0.230667 -0.310648
 -0.190851 -0.331652
 -0.153049 -0.327579
 -0.111008 -0.329433
 -0.073305 -0.319690
 -0.034101 -0.309388
 0.001272 -0.296296
 0.041335 -0.269243
 0.076442 -0.250695
 0.115617 -0.220934
 0.147460 -0.194072
 0.178319 -0.168364
 0.206351 -0.139130
 0.240469 -0.106800
 0.271347 -0.072332
 0.295185 -0.036668
 0.316570 -0.007542
 0.332312 0.036899
 0.341382 0.072222
 0.347040 0.114227
 0.347204 0.154282
 0.340018 0.194286
 0.328212 0.227569
};
\addplot[thin,dotted,smooth
] table {
 0.322533 0.211576
 0.292893 0.243817
 0.272787 0.278424
 0.236153 0.305666
 0.208412 0.320212
 0.174830 0.335251
 0.121229 0.350452
 0.081516 0.341190
 0.032062 0.329059
 -0.010868 0.289830
 -0.049990 0.281521
 -0.078747 0.244272
 -0.105936 0.237544
 -0.137645 0.210857
 -0.167836 0.184740
 -0.194567 0.156215
 -0.223268 0.129592
 -0.252072 0.092751
 -0.273840 0.067913
 -0.298849 0.032085
 -0.324833 -0.005740
 -0.337580 -0.037717
 -0.347669 -0.076298
 -0.362314 -0.135908
 -0.343680 -0.166765
 -0.335773 -0.242397
 -0.301666 -0.267233
 -0.265551 -0.294256
 -0.225447 -0.314801
 -0.178806 -0.326221
 -0.140171 -0.329065
 -0.102668 -0.330485
 -0.065606 -0.318679
 -0.034943 -0.307616
 0.020399 -0.291262
 0.048460 -0.247534
 0.087518 -0.242820
 0.121488 -0.211240
 0.146363 -0.183555
 0.175984 -0.163038
 0.197340 -0.132665
 0.226063 -0.103758
 0.255827 -0.070307
 0.277654 -0.028445
 0.298654 -0.028549
 0.331987 0.035732
 0.336407 0.072201
 0.346142 0.106590
 0.340806 0.145492
 0.328387 0.180242
 0.322533 0.211576
};
\addplot[thin,dotted,smooth
] table {
 0.308601 0.214449
 0.288180 0.236309
 0.270955 0.278382
 0.228297 0.301192
 0.205058 0.318314
 0.169724 0.329074
 0.109863 0.347030
 0.074483 0.334100
 0.033987 0.316723
 -0.020145 0.287997
 -0.053173 0.277730
 -0.090248 0.242250
 -0.108845 0.239658
 -0.139171 0.211053
 -0.166569 0.183110
 -0.195500 0.156843
 -0.224211 0.131548
 -0.236852 0.092441
 -0.273648 0.075463
 -0.297684 0.036614
 -0.324337 -0.003555
 -0.332139 -0.036897
 -0.345262 -0.066572
 -0.357401 -0.109029
 -0.350280 -0.161790
 -0.332071 -0.237863
 -0.299292 -0.266570
 -0.267569 -0.295649
 -0.204915 -0.311097
 -0.174649 -0.326527
 -0.139507 -0.327666
 -0.104457 -0.330379
 -0.051607 -0.314436
 -0.047939 -0.298780
 0.018606 -0.284955
 0.057253 -0.248031
 0.087893 -0.241160
 0.120269 -0.210059
 0.147096 -0.184465
 0.173255 -0.161010
 0.198464 -0.135130
 0.223239 -0.100586
 0.256649 -0.074492
 0.281000 -0.033040
 0.287302 -0.024063
 0.329919 0.040749
 0.338026 0.069548
 0.343355 0.109607
 0.334980 0.167975
 0.325865 0.175938
 0.308601 0.214449
};
\addplot[thin,smooth
] table {
 0.311367 0.217387
 0.296720 0.209836
 0.266365 0.275079
 0.227674 0.302271
 0.200940 0.316936
 0.166889 0.328561
 0.114303 0.347336
 0.042859 0.328348
 0.038008 0.306921
 -0.022093 0.288297
 -0.061417 0.275100
 -0.101392 0.238663
 -0.114804 0.236823
 -0.164398 0.196422
 -0.167351 0.181588
 -0.196576 0.155804
 -0.224773 0.130247
 -0.242174 0.103937
 -0.261129 0.084154
 -0.278075 0.039268
 -0.321230 -0.000488
 -0.334852 -0.032878
 -0.342293 -0.063433
 -0.351264 -0.088019
 -0.326998 -0.155360
 -0.293367 -0.206861
 -0.273002 -0.242552
 -0.267523 -0.289014
 -0.209080 -0.314919
 -0.168801 -0.325935
 -0.140197 -0.328209
 -0.104759 -0.329282
 -0.053538 -0.309055
 -0.038452 -0.306838
 0.035411 -0.267812
 0.061515 -0.247004
 0.088910 -0.238699
 0.122462 -0.207439
 0.148696 -0.183246
 0.175005 -0.167833
 0.207612 -0.141623
 0.214918 -0.099504
 0.255719 -0.076182
 0.281469 -0.031680
 0.287509 -0.022522
 0.330188 0.037818
 0.334745 0.071043
 0.345139 0.108529
 0.329752 0.164018
 0.323723 0.180110
 0.311367 0.217387
};
\end{axis}
\end{tikzpicture}
\caption{The behavior of numerical solution by our method for the mean curvature flow for $t\ge0.09$ with uniform time step size $\Delta t=0.01$.}
\label{fig:MCF_later_times}
\end{figure}
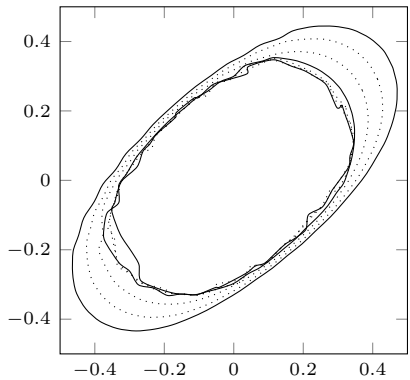

% We used an adaptive method for the time increments of our method and a uniformly small value for the time discretization of the Runge-Kutta method.
% Here, we confirm that our method does not fail even when using large uniform time increments to cause the Runge-Kutta method to fail.
% The initial value is that in Figure \ref{fig:initial_vertices}(a), i.e., uniformly distributed vertices and the time increments are $0.01$ for both our method and the Runge--Kutta method.
% We show the results of the numerical computations in Figure~\ref{fig:MCF_large_time_step}.
% \input{MCF_large_time_step}
% In the Runge--Kutta method, the solution after one step, represented by the dotted line, is already very far from the real solution, confirming that it collapses.
% On the other hand, we can see that our method is stable even with a large time step size, and the behavior of our method is roughly consistent with the previous computations.
% However, we should note that if we try to proceed to later times, we can only do so with an adaptive time step size that considers the time variation of the length.

\subsubsection*{On initial arrangements}

Recall that, as shown in Figure \ref{fig:initial_vertices}, arranging the polygon's vertices that approximate the initial curve uniformly makes the initial distribution slightly different from the original curve.
Below, we check if the rearrangement is mandatory.
The results without the rearrangement are shown in Figure \ref{fig:MCF_without_initial_redistribution}.
We compere them with those with the rearrangement (see Figure~\ref{fig:MCF_comparison}).
It is observed that the results with and without the arrangement are almost identical (for both the proposed method and the Runge--Kutta method).
%Therefore, the rearrangement may not be mandatory for this initial curve, which is not so complicated.
%Therefore, it is not indispensable to make the initial distribution uniform, but it is useful when the numerical computation does not work well.
%An undesired behavior may occur when, for example, the shape of the initial curve is complicated.
%However, we still recommend that the initial distribution is rearranged uniformly so that the numerical computation inherits the property of the enclosed area with good accuracy and avoids an undesired collapse of the numerical solution (see Theorem \ref{thm:dtdL_dtdA_pmbp} and Corollary \ref{cor:dtdL_dtdA_pbmp_audm}).
Therefore, it is not indispensable to make the initial distribution uniform.
However, an undesired behavior may occur when, for example, the shape of the initial curve is complicated.
For such a case, the rearrangement could contribute to inherit the property of the enclosed area with good accuracy and perhaps ensures stability in some sense (see Theorem \ref{thm:dtdL_dtdA_pmbp} and Corollary \ref{cor:dtdL_dtdA_pbmp_audm}).
See \ref{sec:area} for a time discretization that inherits the time evolution law of the enclosed area.
%Summarizing the above, it is not indispensable to make the initial distribution uniform, but it is useful when the numerical computation does not work well, such as when the shape of the initial curve is complex.
\input{MCF_without_initial_redistribution}

% All of the numerical experiments so far have been performed after rearranging the polygon's vertices that approximate the initial curve uniformly.
% As a result, as shown in Figure \ref{fig:initial_vertices}, the vertices after uniform distribution deviate slightly from the original curve.
% Here, we confirm that our method and the Runge-Kutta method work well without uniform distribution of the initial vertices and that the behavior of the solution is almost identical.
% \input{MCF_without_initial_redistribution}
% Figure \ref{fig:MCF_without_initial_redistribution} shows the results of our method and the results of the Runge-Kutta method when there is no uniform distribution at the initial time, i.e., the arrangement in Figure \ref{fig:initial_vertices}(b) is used as the initial value.
% Therefore, because of the time evolution law for the enclosed area (Theorem \ref{thm:dtdL_dtdA_pmbp} and Corollary \ref{cor:dtdL_dtdA_pbmp_audm}) for the semi-discrete problem, we can expect that, given the initial arrangement of vertices, it is preferable to place the initial vertices uniformly using the uniform distribution method before computing the time evolution.
% See \ref{sec:area} for a time discretization focusing on the time evolution law of the enclosed area.

\section{Concluding remarks}
\label{sec:concluding_remarks}
    
In this paper, we proposed the fully discrete polygonal evolution law that inherits the curve-shortening property for planar moving boundary problems. 
The key to the derivation is to devise the definitions of tangential velocities, normal velocities, tangent vectors and normal vectors at each vertex in an implicit manner. 
The proposed evolution law exhibited qualitatively nice behavior even if relatively large time step sizes were employed. 
In particular, the curve-shortening property was corroborated numerically. 
Though only three flows were considered as illustrative examples in this paper, the proposed method is applicable to other flows such as the Helfrich flow.
    
We note several directions for future work. It would be interesting to extend our method to higher dimensional cases and more smooth boundary curves. From the numerical analysis viewpoints, we are also currently attempting to achieve a higher order temporal discretization.

\section*{Acknowledgments}
This work was supported by JSPS KAKENHI No.~18K13455 (KS) and No.~16K17550 (YM), and JST ACT-I Grant Number JPMJPR18US (YM).

\appendix
    
\section{An area-preserving/dissipative fully discrete polygonal evolution law for \eqref{eq:evolution_law}}
\label{sec:area}
\setcounter{alg}{0}
\renewcommand{\thealg}{\Alph{section}.\arabic{alg}}
    
We show that an area-preserving/dissipative fully discrete polygonal evolution law can be constructed by applying a canonical Runge--Kutta method to the semi-discrete evolution law \eqref{eq:polygonal_evolution_law}.

The polygonal evolution law \eqref{eq:polygonal_evolution_law} can be written as a system of ordinary differential equations:
\begin{equation}
	\frac{\mathrm{d}\bm{X}^t}{\mathrm{d}t}
	=
	\bm{F}(\bm{X}^t),
	\quad
	t>0,
	\label{eq:ODE_RK}
\end{equation}
where
\begin{align}
	\bm{X}^t
	&:=
	((\bm{X}_1^t)^{\mathrm{T}},\ldots,(\bm{X}_N^t)^{\mathrm{T}})^{\mathrm{T}}\in\mathbb{R}^{2N},\\
	\bm{F}(\bm{X}^t)
	&:=
	(\bm{f}_1(\bm{X}^t)^{\mathrm{T}},\ldots,\bm{f}_N(\bm{X}^t)^{\mathrm{T}})^{\mathrm{T}}\in\mathbb{R}^{2N},
	\quad
	\bm{f}_i(\bm{X}^t)
	:=
	V_i^t\bm{N}_i^t+W_i^t\bm{T}_i^t\in\mathbb{R}^2.
	\label{eq:ODE_RK_functions}
\end{align}
In this section, the unit tangent and the unit outward normal vectors on vertices are defined
in the same manner to Section \ref{sec:polygonal_moving_boundary_problem}.
% similarly for the semi-discrete polygonal moving boundary problem (see Section \ref{sec:polygonal_moving_boundary_problem}).
For the solution to an $s$-stage Runge--Kutta method
\begin{equation}
	\begin{split}
		\bm{Y}_i
		&=
		(\bm{Y}_{i,1}^{\mathrm{T}},\ldots,\bm{Y}_{i,N}^{\mathrm{T}})^{\mathrm{T}}
		=
		\bm{X}^{n}+\Delta t^n\sum_{j=1}^sa_{ij}\bm{F}(\bm{Y}_j),
		\quad
		i=1,2,\ldots,s,\\
		\bm{X}^{n+1}
		&=
		\bm{X}^{n}+\Delta t^n\sum_{i=1}^sb_i\bm{F}(\bm{Y}_i),
	\end{split}
	\label{eq:RK}
\end{equation}
we have
\begin{align}
	&\frac{\mathcal{A}[\Omega^{n+1}]-\mathcal{A}[\Omega^{n}]}{\Delta t^n}
	=
    \frac{1}{2}\sum_{i=1}^N\frac{J\bm{X}_{i-1}^{n+1}\cdot\bm{X}_i^{n+1}-J\bm{X}_{i-1}^n\cdot\bm{X}_i^{n}}{\Delta t^n}\\
	&\quad=
	\frac{1}{2}\sum_{i=1}^N\left[\sum_{k=1}^sb_kJ(\bm{Y}_{k,i-1}-\bm{Y}_{k,i+1})\cdot\bm{f}_i(\bm{Y}_k)\right.\\
	&\hspace{50pt}\left.+\Delta t^n\sum_{j,k=1}^s(b_kb_j-b_ja_{jk}-b_ka_{kj})J\bm{f}_{i-1}(\bm{Y}_k)\cdot\bm{f}_i(\bm{Y}_j)\right]\\
	&\quad=
	\sum_{k=1}^sb_k\sum_{i=1}^Nv_i(\bm{Y}_k)r_i(\bm{Y}_k)+\sum_{k=1}^sb_k\mathrm{err}_{\mathcal{A},k}\\
	&\hspace{50pt}+\frac{\Delta t^n}{2}\sum_{j,k=1}^s(b_kb_j-b_ja_{jk}-b_ka_{kj})J\bm{f}_{i-1}(\bm{Y}_k)\cdot\bm{f}_i(\bm{Y}_j),
    \label{eq:time_difference_area_RK0}
\end{align}
where
\begin{equation}
	\mathrm{err}_{\mathcal{A},k}
	:=
	\sum_{i=1}^N\left(W_i(\bm{Y}_k)s_i(\bm{Y}_k)-\frac{v_{i+1}(\bm{Y}_k)-v_i(\bm{Y}_k)}{2}\right)\frac{r_{i+1}(\bm{Y}_k)-r_i(\bm{Y}_k)}{2},
	\label{eq:time_difference_area_RK}
\end{equation}
where $s_i(\bm{Y}_k)=\sin\varphi_{i,k}/2$ and $\varphi_{i,k}$ denotes the signed angle between the $i$th edge and the $(i+1)$th edge of $\Gamma(\bm{Y}_k)$.
Note that the last equality in \eqref{eq:time_difference_area_RK0} follows from a similar calculation to \eqref{par_A_Omega}.
% Note the last equality in \eqref{eq:time_difference_area_RK0} is shown in the same way as the method that proved the formula for the time derivative of the enclosed area in Theorem \ref{thm:dtdL_dtdA_pmbp}.
    
If the distribution of initial vertices is uniform and we apply the asymptotic uniform distribution method for computing tangential velocities, then $\mathrm{err}_{\mathcal{A},k}\equiv0$ for all $k$.
Moreover, if we assume the condition
\begin{equation}
	b_ja_{jk}+b_ka_{kj}=b_kb_j,
	\quad
	j,k=1,2,\ldots,s,
	\label{eq:symplecticity}
\end{equation}
then the last term in \eqref{eq:time_difference_area_RK0} also vanishes.
Summarizing the above, we obtain the following theorem.

\begin{thm}
Suppose that the initial vertices are allocated uniformly, i.e. $r_i^0=\text{const.}$ and that the tangential velocities are set based on the asymptotic uniform distribution method.
Then, for the polygonal curve $\Gamma^n$ that evolves according to an $s$-stage Runge--Kutta method applied to \eqref{eq:ODE_RK} with the coefficients satisfying \eqref{eq:symplecticity}, it follows that
%Let $\Gamma^n=\bigcup_{i=1}^N\Gamma_i^n$ denote a polygonal Jordan curve at the $n$th time step and suppose that its time evolution is governed by the polygonal evolution law \eqref{eq:ODE_RK}.
%Moreover, suppose that the initial arrangement of the vertices is uniform.
%For the solution to an $s$-stage Runge--Kutta method that satisfies the condition \eqref{eq:symplecticity}, we have
\begin{equation}
	\frac{\mathcal{A}[\Omega^{n+1}]-\mathcal{A}[\Omega^{n}]}{\Delta t^n}
	=
	\sum_{k=1}^sb_k\sum_{i=1}^Nv_i(\bm{Y}_k)r_i(\bm{Y}_k).
\end{equation}
Namely, a discrete analogue of the formula for the time derivative of the enclosed area holds.
\end{thm}
    
\begin{rem}
A Runge--Kutta method whose coefficients satisfy the condition \eqref{eq:symplecticity} is called canonical or symplectic~(see, e.g.~\cite{bu16,HLW06,SC94}). 
Such a Runge--Kutta method preserves any quadratic invariants of ordinary differential equations~\cite{CO87}.
In this sense, the above theorem looks obvious because the area is a quadratic quantity. 
But we have presented the calculation \eqref{eq:time_difference_area_RK0} to discuss the effect of tangential velocities.
    
The simplest example of a canonical Runge--Kutta method is the mid-point rule: $s=1$, $a_{11} = 1/2$ and $b_1=1$.
Since the second-order convergence is usually expected for the mid-point rule, the rule could be a suitable choice when spatial discretization is of second-order.
\end{rem}
    
\bibliographystyle{plain}
\bibliography{references.bib}
\end{document}